\documentclass{article}

\usepackage[left=1.25in,top=1.25in,right=1.25in,bottom=1.25in,head=1.25in]{geometry}
\usepackage{amsfonts,amsmath,amssymb,amsthm}
\usepackage{verbatim,float,url}
\usepackage{graphicx,subfigure,psfrag}
\usepackage{dsfont,bm}
\usepackage{harvard}
\citationmode{abbr}

\newtheorem{algorithm}{Algorithm}

\newtheorem{lemma}{Lemma}

\theoremstyle{remark}

\theoremstyle{definition}

\newcommand{\argmin}{\mathop{\mathrm{argmin}}}
\newcommand{\argmax}{\mathop{\mathrm{argmax}}}

\def\P{\mathrm{P}}

\def\half{\frac{1}{2}}

\def\col{\mathrm{col}}
\def\row{\mathrm{row}}
\def\nul{\mathrm{null}}
\def\rank{\mathrm{rank}}

\def\sign{\mathrm{sign}}
\def\supp{\mathrm{supp}}
\def\diag{\mathrm{diag}}
\def\aff{\mathrm{aff}}

\def\hbeta{\hat{\beta}}
\def\tbeta{\tilde{\beta}}
\def\lone{1}
\def\ltwo{2}
\def\linf{\infty}

\def\T{^T}
\def\R{\mathds{R}}
\def\cA{\mathcal{A}}

\def\cE{\mathcal{E}}

\def\cN{\mathcal{N}}

\def\LARS{\mathrm{LARS}}

\title{The Lasso Problem and Uniqueness}
\author{Ryan J. Tibshirani}
\date{\it Carnegie Mellon University}

\begin{document}
\maketitle

\begin{abstract}
The lasso is a popular tool for sparse linear 
regression, especially for problems in which the number of
variables $p$ exceeds the number of observations $n$. But
when $p>n$, the lasso criterion is not strictly convex, and hence
it may not have a unique minimum. An important question is:
when is the lasso solution well-defined (unique)? We review results 
from the literature, which show that
if the predictor variables are drawn from a continuous probability
distribution, then there is a unique lasso solution with
probability one, regardless of the sizes of $n$ and $p$.
We also show that this result extends easily to $\ell_1$
penalized minimization problems over a wide range of loss
functions. 

A second important question is: how can we manage the case of 
non-uniqueness in lasso solutions? In light of the aforementioned result, 
this case really only arises when some of the predictor 
variables are discrete, or when some post-processing has been performed on 
continuous predictor measurements. Though we certainly cannot claim to provide 
a complete answer to such a broad question, we do present progress towards 
understanding some aspects of non-uniqueness. First, we extend the LARS algorithm 
for computing the lasso solution path to cover the non-unique case, so that this 
path algorithm works for any predictor matrix. Next, we derive a simple method for 
computing the component-wise uncertainty in lasso solutions of any given problem 
instance, based on linear programming. Finally, we review results from the 
literature on some of the unifying properties of lasso solutions, 
and also point out particular forms of solutions that have distinctive 
properties.
\end{abstract}

\section{Introduction}
\label{sec:intro}

We consider $\ell_1$ penalized linear regression, 
also known as the lasso problem \cite{lasso,bp}. Given a response vector
$y \in \R^n$, a matrix $X \in \R^{n\times p}$ of predictor variables, and a 
tuning parameter $\lambda \geq 0$, the lasso estimate can be defined as
\begin{equation}
\label{eq:lasso}
\hbeta \in \argmin_{\beta \in \R^p} \, 
\half \|y-X\beta\|_\ltwo^2 + \lambda \|\beta\|_\lone.
\end{equation}
The lasso solution is unique when $\rank(X)=p$, because the criterion is strictly
convex. This is not true when $\rank(X)<p$, and in this case, there can be 
multiple minimizers of the lasso criterion (emphasized by the element notation 
in \eqref{eq:lasso}). Note that when the number of variables
exceeds the number of observations, $p>n$, we must have $\rank(X)<p$.

The lasso is quite a popular tool for estimating the coefficients in
a linear model, especially in the high-dimensional setting, $p>n$. Depending
on the value of the tuning parameter $\lambda$, solutions of the lasso problem
will have many coefficients set exactly to zero, due to the 
nature of the $\ell_1$ penalty. We tend to think of the support set of a lasso 
solution $\hbeta$, written $\cA=\supp(\hbeta) \subseteq \{1,\ldots p\}$ and 
often referred to as the active set, as describing 
a particular subset of important variables for the linear model
of $y$ on $X$. Recently, there has been a lot of interesting work legitimizing this
claim by proving 
desirable properties of $\hbeta$ or its active set $\cA$, in terms of estimation
error or model recovery. Most of this work falls into the setting $p>n$. But 
such properties are not the focus of the 
current paper. Instead, our focus somewhat simpler, and at somewhat more of a
basic level: we investigate issues concerning the uniqueness or non-uniqueness 
of lasso solutions.

Let us first take a step back, and consider the usual linear 
regression estimate (given by $\lambda=0$ in \eqref{eq:lasso}), as
a motivating example.
Students of statistics are taught to distrust the coefficients given by 
linear regression when $p>n$. We may ask: why? Arguably, the main reason
is that the linear regression solution is not unique when $p>n$ (or
more precisely, when $\rank(X)<p$), and further, this non-uniqueness occurs
in such a way that we can always find a variable $i \in \{1,\ldots p\}$ whose 
coefficient is positive at one solution and negative at another. 
(Adding any element of the null space of $X$ to one least squares
solution produces another solution.) 
This makes it generally impossible to interpret the linear regression
estimate when $p>n$.

Meanwhile, the lasso estimate is also not unique when $p>n$ (or when 
$\rank(X)<p$), but it is commonly used in 
this case, and in practice little attention is paid to uniqueness. 
Upon reflection, this seems somewhat surprising, because non-uniqueness of 
solutions can cause major problems in terms of interpretation (as demonstrated 
by the linear regression case). Two basic questions are:
\begin{itemize}
\item Do lasso estimates suffer from the same sign inconsistencies as
do linear regression estimates? That is, for a fixed $\lambda$, 
can one lasso solution have a positive $i$th coefficient, 
and another have a negative $i$th coefficient?
\item Must any two lasso solutions, at the same value of $\lambda$, necessarily 
share the same support, and differ only in their estimates of the nonzero 
coefficient values? Or can different lasso solutions exhibit different active 
sets? 
\end{itemize}
Consider the following example, concerning the second question. Here we 
let $n=5$ and $p=10$. For a particular response $y \in \R^5$ and predictor
matrix $X \in \R^{5 \times 10}$, and $\lambda=1$, we found two solutions 
of the lasso problem \eqref{eq:lasso}, using two different algorithms. 
These are 
\begin{align*}
\hbeta^{(1)} &=  (-0.893, 0.620, 0.375, 0.497, \ldots, 0)\T
\;\,\text{and} \\
\hbeta^{(2)} &=  (-0.893, 0.869, 0.624, 0, \ldots, 0)\T,
\end{align*}
where we use ellipses to denote all zeros. In other words, the
first solution has support set $\{1,2,3,4\}$, and the second has
support set $\{1,2,3\}$. This is not at all ideal for the purposes 
of interpretation,
because depending on which algorithm we used to minimize the lasso
criterion, we may have considered the 4th variable to be important 
or not. Moreover, who knows which variables may have zero coefficients at
other solutions?

In Section \ref{sec:unique}, we show that if the entries of the predictor matrix 
$X$ are drawn from a continuous probability distribution, then we essentially never
have to worry about the latter problem---along with the problem of sign 
inconsistencies, and any other issues relating to non-uniqueness---because the 
lasso solution is unique with probability one. We emphasize that here uniqueness is 
ensured with probability one (over the distribution of $X$) regardless of the sizes 
of $n$ and $p$. This result has basically appeared in various forms in the 
literature, but is perhaps not as well-known as it should be. Section 
\ref{sec:unique} gives a detailed review of why this fact is true. 

Therefore, the two questions raised above only need to be addressed in the case
that $X$ contains discrete predictors, or contains some kind of post-processed 
versions of continuously drawn predictor measurements. To put it bluntly (and save
any dramatic tension), the 
answer to the first question is ``no''. In other words, no two lasso solutions 
can attach opposite signed coefficients to the same variable. We show
this using a very simple argument in Section \ref{sec:coefbounds}. As for the second
question, the example above already shows that the answer is unfortunately ``yes''.
However, the multiplicity of active sets can be dealt with in a principled manner,
as we argue in Section \ref{sec:coefbounds}. Here we show how to compute
lower and upper bounds on the coefficients of lasso solutions of any particular
problem instance---this reveals exactly which variables are assigned zero coefficients
at some lasso solutions, and which variables have nonzero coefficients at all lasso 
solutions.

Apart from addressing these two questions, we also attempt to better understand
the non-unique case through other means. In Section \ref{sec:lars}, we extend the
well-known LARS algorithm for computing the lasso solution path (over the 
tuning parameter $\lambda$) to cover the 
non-unique case. Therefore the (newly proposed) LARS algorithm can compute a lasso 
solution path for any predictor matrix $X$. (The existing LARS algorithm cannot,
because it assumes that for any $\lambda$ the active variables form 
a linearly independent set, which is not true in general.) The special lasso 
solution computed by the LARS algorithm, also called the LARS lasso solution, possesses
several interesting properties in the non-unique case. We explore these mainly in 
Section \ref{sec:lars}, and to a lesser extent in Section \ref{sec:relprops}.
Section \ref{sec:relprops} contains a few final miscellaneous 
properties relating to non-uniqueness, and the work of the previous three 
sections.

In this paper, we both review existing results from the literature,
and establish new ones, on the topic of uniqueness of lasso solutions.
We do our best to acknowledge existing works in the literature, with citations
either immediately preceeding or succeeding the statements of lemmas. 
The contents of this paper were already discussed above, but this was presented
out of order, and hence we give
a proper outline here. We begin in Section 
\ref{sec:unique} by examining the KKT optimality conditions for the lasso
problem, and we use these to derive sufficient conditions for the uniqueness
of the lasso solution. This culminates in a result that says that if the
entries of $X$ are continuously distributed, then the lasso solution is 
unique with probability one. We also show that this same result holds for
$\ell_1$ penalized minimization problems over a broad class of loss functions.
Essentially, the rest of the paper focuses on the case of a non-unique lasso
solution.
Section \ref{sec:lars} presents an extension of the LARS algorithm
for the lasso solution path that works for any predictor matrix $X$ 
(the original LARS algorithm really only applies to the case of a unique 
solution). We then discuss some special properties of the
LARS lasso solution. Section \ref{sec:coefbounds} develops a method for
computing component-wise lower and upper bounds on lasso coefficients
for any given problem instance. In Section \ref{sec:relprops}, we finish
with some related properties, concerning the different active sets of 
lasso solutions, and a necessary condition for uniqueness. Section 
\ref{sec:discuss} contains some discussion.

Finally, our notation in the paper is as follows. For a matrix $A$,
we write $\col(A)$, $\row(A)$, and $\nul(A)$ to denote its column space, 
row space, and null space, respectively. We use $\rank(A)$ for the rank
of $A$. We use $A^+$ to denote the Moore-Penrose pseudoinverse of $A$,
and when $A$ is rectangular, this means $A^+= (A\T A)^+ A\T$. For a
linear subspace $L$, we write $P_L$ for the projection map onto $L$.
Suppose that $A \in \R^{n\times p}$ has columns $A_1,\ldots A_p \in \R^n$, 
written $A=[A_1,\ldots A_p]$. Then for an index set $S = \{i_1,\ldots i_k\} 
\subseteq \{1,\ldots p\}$, we let $A_S = [A_{i_1},\ldots A_{i_k}]$;
in other words, $A_S$ extracts the columns of $A$ in $S$. Similarly,
for a vector $b \in \R^p$, we let $b_S = (b_{i_1},\ldots b_{i_r})\T$,
or in other words, $b_S$ extracts the components of $b$ in $S$. We
write $A_{-S}$ or $b_{-S}$ to extract the columns or components not 
in $S$.

\section{When is the lasso solution unique?}
\label{sec:unique}

In this section, we review the question: when is the lasso solution unique? 
In truth, we only give a partial answer, because we provide sufficient 
conditions for a unique minimizer of the lasso criterion. Later, in Section
\ref{sec:relprops}, we study the other direction (a necessary condition for 
uniqueness).

\subsection{Basic facts and the KKT conditions}
\label{sec:basic}

We begin by recalling a few basic facts about lasso solutions.

\begin{lemma}
\label{lem:basic}
For any $y,X$, and $\lambda \geq 0$, the lasso problem 
\eqref{eq:lasso} has the following properties:
\begin{itemize}
\item[(i)] There is either a unique lasso solution  
or an (uncountably) infinite number of solutions.
\item[(ii)] Every lasso solution $\hbeta$ gives the same 
fitted value $X\hbeta$. 
\item[(iii)] If $\lambda>0$, then every lasso solution $\hbeta$ 
has the same $\ell_1$ norm, $\|\hbeta\|_\lone$. 
\end{itemize}
\end{lemma}

\begin{proof}
(i) The lasso criterion is convex and has no directions
of recession (strictly speaking, when $\lambda=0$ the criterion
can have directions of recession, but these are directions in 
which the criterion is constant). Therefore it attains its 
minimum over $\R^p$ (see, for example, Theorem 27.1 of 
\citeasnoun{rockafellar}), that is, the lasso problem has at 
least one solution. Suppose now that there are two solutions
$\hbeta^{(1)}$ and $\hbeta^{(2)}$, $\hbeta^{(1)} \not=
\hbeta^{(2)}$. Because the solution set of a convex minimization
problem is convex, we know that 
$\alpha\hbeta^{(1)}+(1-\alpha)\hbeta^{(2)}$ is also a
solution for any $0 < \alpha < 1$, which gives 
uncountably many lasso solutions as $\alpha$ varies over 
$(0,1)$. 

\smallskip

(ii) Suppose that we have two solutions $\hbeta^{(1)}$ and
$\hbeta^{(2)}$ with $X\hbeta^{(1)} \not= X\hbeta^{(2)}$. Let
$c^*$ denote the minimum value of the lasso criterion obtained
by $\hbeta^{(1)},\hbeta^{(2)}$. For any $0 < \alpha < 1$, we 
have
\begin{equation*}
\half\|y-X(\alpha\hbeta^{(1)}+(1-\alpha)\hbeta^{(2)})\|_\ltwo^2
+ \lambda \|\alpha\hbeta^{(1)}+(1-\alpha)\hbeta^{(2)}\|_\lone
< \alpha c^* + (1-\alpha) c^* = c^*,
\end{equation*}
where the strict inequality is due to the strict convexity of
the function $f(x)=\|y-x\|_\ltwo^2$ along with the convexity of
$f(x)=\|x\|_\lone$. This means that 
$\alpha\hbeta^{(1)}+(1-\alpha)\hbeta^{(2)}$ attains a lower
criterion value than $c^*$, a contradiction.

\smallskip

(iii) By (ii), any two solutions  
must have the same fitted value, and hence
the same squared error loss. But the solutions also attain the 
same value of the lasso criterion, and if $\lambda>0$, then 
they must have the same $\ell_1$ norm.
\end{proof}

To go beyond the basics, we turn to the Karush-Kuhn-Tucker
(KKT) optimality conditions for the lasso problem \eqref{eq:lasso}. 
These conditions can be written as
\begin{gather}
\label{eq:lassokkt}
X\T (y-X\hbeta) = \lambda \gamma, \\
\label{eq:lassosg}
\gamma_i \in \begin{cases}
\{\sign(\hbeta_i)\} & \text{if} \;\; \hbeta_i \not= 0 \\
[-1,1] & \text{if} \;\; \hbeta_i = 0
\end{cases}, \;\;\;\text{for}\;\, i=1,\ldots p.
\end{gather}
Here $\gamma \in \R^p$ is called a subgradient of the function 
$f(x) = \|x\|_\lone$ evaluated at $x=\hbeta$. Therefore
$\hbeta$ is a solution in \eqref{eq:lasso} if and only if $\hbeta$ 
satisfies \eqref{eq:lassokkt} and \eqref{eq:lassosg} for some 
$\gamma$. 

We now use the KKT conditions to write the lasso fit and solutions in
a more explicit form. In what follows, we assume that $\lambda>0$ for the 
sake of simplicity (dealing with the case $\lambda=0$ is not difficult, 
but some of the definitions and statements need to be modified, 
avoided here in order to preserve readibility). 
First we define the 
equicorrelation set $\cE$ by
\begin{equation}
\label{eq:equiset}
\cE = \big\{i \in \{1,\ldots p\} : |X_i\T (y-X\hbeta)|=
\lambda\big\}.
\end{equation}
The equicorrelation set $\cE$ is named as such because when $y,X$ 
have been standardized, $\cE$ contains the variables 
that have equal (and maximal) absolute correlation with the 
residual. We define the equicorrelation signs $s$ by
\begin{equation}
\label{eq:equisigns}
s = \sign\big(X_\cE\T(y-X\hbeta)\big).
\end{equation}
Recalling \eqref{eq:lassokkt}, we note that the optimal subgradient
$\gamma$ is unique (by the uniqueness of the fit $X\hbeta$),
and we can equivalently define $\cE,s$ in terms of
$\gamma$, as in $\cE=\{i \in \{1,\ldots p\} : |\gamma_i|=1\}$ and 
$s=\gamma_\cE$. 
The uniqueness of $X\hbeta$ (or the uniqueness of $\gamma$) implies 
the uniqueness of $\cE,s$.

By definition of the subgradient $\gamma$ in \eqref{eq:lassosg}, we know 
that $\hbeta_{-\cE}=0$ for any lasso solution $\hbeta$.
Hence the $\cE$ block of \eqref{eq:lassokkt} can be 
written as
\begin{equation}
\label{eq:lassokkte}
X_\cE\T (y-X_\cE \hbeta_\cE) = \lambda s.
\end{equation}
This means that $\lambda s \in \row(X_\cE)$, so $\lambda s = 
X_\cE\T(X_\cE\T)^+ \lambda s$. Using this fact, and rearranging
\eqref{eq:lassokkte}, we get
\begin{equation*}
X_\cE\T X_\cE \hbeta_\cE = X_\cE\T\big(y - 
(X_\cE\T)^+ \lambda s\big).
\end{equation*}
Therefore the (unique) lasso fit 
$X\hbeta = X_\cE \hbeta_\cE$ is
\begin{equation}
\label{eq:lassofit}
X\hbeta = X_\cE (X_\cE)^+ \big(y -
(X_\cE\T)^+ \lambda s \big),
\end{equation}
and any lasso solution $\hbeta$ is of the form
\begin{equation}
\label{eq:lassosol}
\hbeta_{-\cE}=0 \;\;\;\text{and}\;\;\;
\hbeta_\cE = (X_\cE)^+ \big(y - (X_\cE\T)^+\lambda s \big) + b,
\end{equation}
where $b \in \nul(X_\cE)$. In particular,
any $b \in \nul(X_\cE)$ produces a lasso solution $\hbeta$ 
in \eqref{eq:lassosol} provided that $\hbeta$ has the correct signs 
over its nonzero coefficients, that is, 
$\sign(\hbeta_i)=s_i$ for all $\hbeta_i\not=0$. 
We can write these conditions together as
\begin{equation}
\label{eq:lassosign}
b \in \nul(X_\cE) \;\;\;\text{and}\;\;\;
s_i \cdot 
\Big(\big[(X_\cE)^+\big(y-(X_\cE\T)^+\lambda s\big)\big]_i + 
b_i\Big) \geq 0 \;\;\;\text{for}\;\, i \in \cE,
\end{equation}
and hence any $b$ satisfying \eqref{eq:lassosign} gives a lasso 
solution $\hbeta$ in \eqref{eq:lassosol}. In the next section, using a 
sequence of straightforward arguments, we prove that the lasso solution is
unique under somewhat general conditions.

\subsection{Sufficient conditions for uniqueness}
\label{sec:uniquesuff}

From our work in the previous section, we can see that if $\nul(X_\cE)=\{0\}$,
then the lasso solution is unique and is given by \eqref{eq:lassosol} with
$b=0$. (We note that $b=0$ necessarily satisfies the sign condition 
in \eqref{eq:lassosign}, because a lasso solution is guaranteed to exist by 
Lemma \ref{lem:basic}.) Then by rearranging \eqref{eq:lassosol}, done to 
emphasize the rank of $X_\cE$, we have the following result.

\begin{lemma}
\label{lem:unique}
For any $y,X$, and $\lambda>0$, if $\nul(X_\cE)=\{0\}$,
or equivalently if $\rank(X_\cE)=|\cE|$, then the lasso solution
is unique, and is given by
\begin{equation}
\label{eq:lassousol}
\hbeta_{-\cE}=0 \;\;\;\text{and}\;\;\;
\hbeta_\cE = (X_\cE\T X_\cE)^{-1} (X_\cE\T y - \lambda s),
\end{equation}
where $\cE$ and $s$ are the equicorrelation set and signs as defined in
\eqref{eq:equiset} and \eqref{eq:equisigns}. Note that this solution has
at most $\min\{n,p\}$ nonzero components. 
\end{lemma}

This sufficient condition for uniqueness has appeared many times in the literature. 
For example, see \citeasnoun{homotopy2}, \citeasnoun{fuchs}, \citeasnoun{sharp}, 
\citeasnoun{nearideal}. We will show later in Section \ref{sec:relprops} 
that the same condition is actually also necessary, for all almost every $y \in \R^n$.

Note that $\cE$ depends on the lasso solution at $y,X,\lambda$, 
and hence the condition $\nul(X_\cE)=\{0\}$ is somewhat circular. There
are more natural conditions, depending on $X$ alone, that imply $\nul(X_\cE)=\{0\}$. 
To see this, suppose that $\nul(X_\cE) \not= \{0\}$; then for some $i \in \cE$, we 
can write
\begin{equation*}
X_i = \sum_{j \in \cE\setminus\{i\}} c_j X_j,
\end{equation*}
where $c_j \in \R$, $j \in \cE\setminus\{i\}$. Hence, 
\begin{equation*}
s_i X_i = \sum_{j \in \cE\setminus\{i\}} (s_i s_j c_j) \cdot (s_j X_j).
\end{equation*}
By definition of the equicorrelation set, $X_j\T r = s_j\lambda$ for any
$j \in \cE$, where $r=y-X\hbeta$ is the lasso residual. Taking the inner 
product of both sides above with $r$, we get 
\begin{equation*}
\lambda = \sum_{j \in \cE\setminus\{i\}} (s_i s_j c_j) \lambda,
\end{equation*}
or
\begin{equation*}
\sum_{j \in \cE\setminus\{i\}} (s_i s_j c_j) = 1,
\end{equation*}
assuming that $\lambda>0$. Therefore, we have shown that
if $\nul(X_\cE)\not=\{0\}$, then for some $i\in\cE$,
\begin{equation*}
s_iX_i = \sum_{j \in \cE\setminus\{i\}} 
a_j \cdot s_jX_j,
\end{equation*}
with $\sum_{j \in \cE\setminus\{i\}} a_j = 1$, which means
that $s_iX_i$ lies in the affine span of $s_jX_j$, $j\in\cE\setminus\{i\}$.
Note that we can assume without a loss of generality that 
$\cE\setminus\{i\}$ has at most $n$ elements, since otherwise we can simply
repeat the above arguments replacing $\cE$ by any one of its subsets with
$n+1$ elements; hence the affine span of $s_jX_j$, $j\in\cE\setminus\{i\}$
is at most $n-1$ dimensional. 

We say that the matrix $X \in \R^{n\times p}$ has columns in {\it general position}
if any affine subspace $L \subseteq \R^n$ of dimension $k<n$ contains 
contains no more than $k+1$ elements of the set $\{\pm X_1,\ldots \pm X_p\}$, 
excluding antipodal pairs. Another way of saying this: the
affine span of any $k+1$ points $\sigma_1X_{i_1},\ldots \sigma_{k+1}X_{i_{k+1}}$, 
for arbitrary
signs $\sigma_1,\ldots \sigma_{k+1} \in \{-1,1\}$, does not contain any element
of $\{\pm X_i: i\not= i_1,\ldots i_{k+1}\}$. From what we have just shown,
the predictor matrix $X$ having columns in general position is enough to ensure
uniqueness.

\begin{lemma}
\label{lem:unique2}
If the columns of $X$ are in general position, then for any $y$ and
$\lambda>0$, the lasso solution is unique and is given by \eqref{eq:lassousol}.
\end{lemma}

This result has also essentially appeared in the literature, taking various forms
when stated for various related problems. For example, \citeasnoun{boostpath} give 
a similar result for general convex loss functions. 
\citeasnoun{dossal} gives a related result for the noiseless lasso problem (also
called basis pursuit). \citeasnoun{donmost} gives results tying togther the 
uniqueness (and equality) of solutions of the noiseless lasso problem and the 
corresponding $\ell_0$ minimization problem.

Although the definition of general position 
may seem somewhat technical, this condition is naturally satisfied when the entries 
of the predictor matrix $X$ are drawn from a continuous probability distribution. 
More precisely, if the entries of $X$ follow a joint distribution that 
is absolutely continuous with respect to Lebesgue measure on $\R^{np}$, then the 
columns of $X$ are in general position with probability one. To see this, first
consider the probability $\P(X_{k+2} \in \aff\{X_1,\ldots X_{k+1}\})$, where 
$\aff\{X_1,\ldots X_{k+1}\}$ denotes the affine span of $X_1,\ldots X_{k+1}$. 
Note that, by continuity, 
\begin{equation*}
\P(X_{k+2} \in \aff\{X_1,\ldots X_{k+1}\} \,|\, 
X_1,\ldots X_{k+1})=0,
\end{equation*}
because (for fixed $X_1,\ldots X_{k+1}$)
the set $\aff\{X_1,\ldots X_{k+1}\} \subseteq \R^n$ has Lebesgue
measure zero. Therefore, integrating over $X_1,\ldots X_{k+1}$, we get that
$\P(X_{k+2} \in \aff\{X_1,\ldots X_{k+1}\})=0$.
Taking a union over all subsets of $k+2$ columns, all 
combinations of $k+2$ signs, and all $k<n$, we conclude that with 
probability zero the columns are not in general position. This leads us
to our final sufficient condition for uniqueness of the lasso solution.

\begin{lemma}
\label{lem:unique3}
If the entries of $X \in \R^{n\times p}$
are drawn from a continuous probability distribution on $\R^{np}$, 
then for any $y$ and $\lambda>0$,
the lasso solution is unique and is given by \eqref{eq:lassousol} with 
probability one. 
\end{lemma}

According to this result, we essentially never have to worry about 
uniqueness when the predictor variables come from a continuous 
distribution, regardless of the sizes of $n$ and $p$. Actually,
there is nothing really special about $\ell_1$ penalized linear
regression in particular---we show next that the same uniqueness result holds 
for $\ell_1$ penalized minimization with any differentiable, strictly convex
loss function.

\subsection{General convex loss functions}

We consider the more general minimization problem
\begin{equation}
\label{eq:gen}
\hbeta \in \argmin_{\beta \in \R^p} \, f(X\beta) + \lambda \|\beta\|_\lone,
\end{equation}
where the loss function $f : \R^n \rightarrow \R$ is differentiable and strictly 
convex. To be clear, we mean that $f$ is strictly convex in its 
argument, so for example the function $f(u)=\|y-u\|_\ltwo^2$ is strictly convex, 
even though $f(X\beta)=\|y-X\beta\|_\ltwo^2$ may not be strictly 
convex in $\beta$. 

The main ideas from Section \ref{sec:basic} carry over to this more general
problem. The arguments given in the proof of Lemma \ref{lem:basic} can 
be applied (relying on the strict convexity of $f$) to show that the same set of 
basic results hold for problem \eqref{eq:gen}: (i) there is either a unique 
solution or uncountably many solutions;\footnote{To be precise, if $\lambda=0$
then problem \eqref{eq:gen} may not have a solution for an arbitrary differentiable,
strictly convex function $f$. This is because $f$ may have directions of recession 
that are not directions in which $f$ is constant, and hence it may not attain its 
minimal value. For example, the function $f(u)=e^{-u}$ is differentiable and
strictly convex on $\R$, but does not attain its minimum. Therefore, for $\lambda=0$,
the statements in this section should all be interpreted as conditional on the existence
of a solution in the first place. For $\lambda>0$, the $\ell_1$ penalty gets rid of 
this issue, as the criterion in \eqref{eq:gen} has no directions of recession, 
implying the existence of a solution.} (ii) every solution $\hbeta$ gives 
the same fit $X\hbeta$; (iii) if $\lambda>0$, then every solution $\hbeta$ has the 
same $\ell_1$ norm. The KKT conditions for \eqref{eq:gen} can be expressed as
\begin{gather}
\label{eq:genkkt}
X\T (-\nabla f)(X\hbeta) = \lambda \gamma, \\
\label{eq:gensg}
\gamma_i \in \begin{cases}
\{\sign(\hbeta_i)\} & \text{if} \;\; \hbeta_i \not= 0 \\
[-1,1] & \text{if} \;\; \hbeta_i = 0
\end{cases}, \;\;\;\text{for}\;\, i=1,\ldots p,
\end{gather}
where $\nabla f : \R^n \rightarrow \R^n$ is the gradient
of $f$, and we can define the equicorrelation set and signs in the
same way as before,
\begin{equation*}
\cE = \big\{i \in \{1,\ldots p\} : |X_i\T (-\nabla f)(X\hbeta)|=
\lambda\big\},
\end{equation*}
and 
\begin{equation*}
s = \sign\big(X_\cE\T (-\nabla f)(X\hbeta)\big).
\end{equation*}
The subgradient condition \eqref{eq:gensg} implies that $\hbeta_{-\cE}=0$
for any solution $\hbeta$ in \eqref{eq:gen}. For squared error loss, recall
that we then explicitly solved for $\hbeta_\cE$ as a function
of $\cE$ and $s$. This is not possible for a general loss function $f$; but
given $\cE$ and $s$, we can rewrite the minimization problem \eqref{eq:gen} 
over the coordinates in $\cE$ as
\begin{equation}
\label{eq:gene}
\hbeta_\cE \in \argmin_{\beta_\cE \in \R^{|\cE|}} \,
f(X_\cE\beta_\cE) + \lambda \|\beta_\cE\|_\lone.
\end{equation}
Now, if $\nul(X_\cE)=\{0\}$ (equivalently $\rank(X_\cE)=|\cE|$), then the 
criterion in \eqref{eq:gene} is strictly convex, as $f$ itself is strictly
convex. This implies that there is a unique solution $\hbeta_\cE$ in
\eqref{eq:gene}, and therefore a unique solution $\hbeta$ in \eqref{eq:gen}.
Hence, we arrive at the same conclusions as those made in Section 
\ref{sec:uniquesuff}, that there is a unique solution in \eqref{eq:gen} if the
columns of $X$ are in general position, and ultimately, the following result.

\begin{lemma}
\label{lem:unique4}
If $X \in \R^{n\times p}$ has entries drawn from a continuous 
probability distribution on $\R^{np}$, then
for any differentiable, strictly convex function $f$, and for any $\lambda>0$,
the minimization problem 
\eqref{eq:gen} has a unique solution with probability one.
This solution has at most $\min\{n,p\}$ nonzero components.
\end{lemma}

This general result applies to any differentiable, strictly convex loss 
function $f$, which is quite a broad class. For example, it applies to 
logistic regression loss,
\begin{equation*}
f(u) = \sum_{i=1}^n \big[ -y_i u_i + \log\big(1+\exp(u_i)\big) \big],
\end{equation*}
where typically (but not necessarily) each $y_i \in \{0,1\}$, 
and Poisson regression loss,
\begin{equation*}
f(u) = \sum_{i=1}^n \big[ -y_i u_i + \exp(u_i)\big],
\end{equation*}
where typically (but again, not necessarily) each $y_i \in \mathds{N} 
= \{0,1,2,\ldots \}$. 

We shift our focus in the next section, and without assuming any conditions for 
uniqueness, we show how to compute a solution path for the lasso problem (over 
the regularization parameter $\lambda$).

\section{The LARS algorithm for the lasso path}
\label{sec:lars}

The LARS algorithm is a great tool for understanding the behavior of 
lasso solutions. (To be clear, here and throughout we use the term ``LARS algorithm'' 
to refer to the version of the algorithm that computes the lasso solution path, and 
not the version that performs a special kind of forward variable selection.)  
The algorithm begins at 
$\lambda=\infty$, where the lasso solution is trivially $0 \in \R^p$.
Then, as the parameter $\lambda$ decreases, it computes a solution path
$\hbeta^\LARS(\lambda)$ that is piecewise linear and 
continuous as a function of $\lambda$. Each knot in this path corresponds 
to an iteration of the algorithm, in which the path's linear trajectory is
altered in order to satisfy the KKT optimality conditions.
The LARS algorithm was proposed (and named) by \citeasnoun{lars}, though 
essentially the same idea appeared earlier in the works of 
\citeasnoun{homotopy1} and \citeasnoun{homotopy2}. 
It is worth noting that the LARS algorithm (as proposed in any of these works) 
assumes that $\rank(X_\cE)=|\cE|$ throughout the lasso path. 
This is not necessarily correct when $\rank(X)<p$, and can lead to errors in 
computing lasso solutions. (However, from what we showed in Section
\ref{sec:unique}, this ``naive'' assumption is indeed correct with probability 
one when the predictors are drawn from a continuous distribution, and this is 
likely the reason why such a small oversight has gone unnoticed since the time 
of the original publications.)

In this section, we extend the LARS algorithm to cover a 
generic predictor matrix $X$.\footnote{The description of 
this algorithm and its proof of correctness previously appeared in Appendix B 
of the author's doctoral dissertation \cite{ryanphd}.}
Though the lasso solution is not necessarily 
unique in this general case, and we may
have $\rank(X_\cE)<|\cE|$ at some points along path, we show that a piecewise 
linear and continuous path of solutions still exists, and computing this 
path requires only a simple modification to the previously proposed LARS
algorithm. We describe the algorithm and its steps in detail, but delay the
proof of its correctness until Appendix \ref{app:larsproof}. We also present
a few properties of this algorithm and the solutions along its path. 


\subsection{Description of the LARS algorithm}
\label{sec:larsalg}

We start with an overview of the LARS algorithm to compute the lasso path
(extended to cover an arbitrary predictor matrix $X$), and then we 
describe its steps in detail at a general iteration $k$. The algorithm 
presented here is of course very similar to the original LARS algorithm of 
\citeasnoun{lars}. The key difference is the following: if
$X_\cE\T X_\cE$ is singular, then the KKT conditions over 
the variables in $\cE$ no longer have a unique solution, and the current
algorithm uses the solution with the minimum $\ell_2$ norm, as in 
\eqref{eq:lassols} and \eqref{eq:lassols2}. This seemingly minor detail is 
the basis for the algorithm's correctness in the general $X$ case.

\begin{algorithm}[\textbf{The LARS algorithm for the lasso path}]
\label{alg:lars}
\hfill\par
\smallskip
\smallskip
Given $y$ and $X$.
\begin{itemize}
\item Start with the iteration counter $k=0$, 
regularization parameter $\lambda_0=\infty$, 
equicorrelation set $\cE=\emptyset$, and 
equicorrelation signs $s=\emptyset$.  
\item While $\lambda_k>0$:
\begin{enumerate}
\item Compute the LARS lasso solution at $\lambda_k$ by least squares,
as in \eqref{eq:lassols} and \eqref{eq:lassols2}. Continue in a
linear direction from the solution for $\lambda \leq \lambda_k$.
\item Compute the next joining time $\lambda^{\mathrm{join}}_{k+1}$,
when a variable outside the equicorrelation
set achieves the maximal absolute inner product with the 
residual, as in \eqref{eq:lassojoin} and \eqref{eq:lassonextj}.  
\item Compute the next crossing time $\lambda^{\mathrm{cross}}_{k+1}$, 
when the coefficient path of an 
equicorrelation variable crosses through zero, as in
\eqref{eq:lassocross} and \eqref{eq:lassonextc}. 
\item Set $\lambda_{k+1}=
\max\{\lambda^{\mathrm{join}}_{k+1},\lambda^{\mathrm{cross}}_{k+1}\}$. 
If $\lambda^{\mathrm{join}}_{k+1} > \lambda^{\mathrm{cross}}_{k+1}$, then 
add the joining variable to $\cE$ and its sign to $s$; otherwise, 
remove the crossing variable from $\cE$ and its sign from $s$. 
Update $k=k+1$.   
\end{enumerate}
\end{itemize}
\end{algorithm}

At the start of the $k$th iteration, the regularization parameter is 
$\lambda=\lambda_k$.
For the path's solution at $\lambda_k$, we set the 
non-equicorrelation coefficients equal to zero, $\hbeta^\LARS_{-\cE}(\lambda_k)=0$,
and we compute the equicorrelation coefficients as
\begin{equation}
\label{eq:lassols}
\hbeta^\LARS_\cE (\lambda_k)  = (X_\cE)^+
\big(y - (X_\cE\T)^+ \lambda_k s\big) = c - \lambda_k d,
\end{equation}
where $c=(X_\cE)^+y$ and $d=(X_\cE)^+(X_\cE\T)^+ s = (X_\cE\T X_\cE)^+ s$ 
are defined to help emphasize that this is a linear function of the regularization 
parameter. This estimate can be viewed as the minimum $\ell_2$ norm solution of 
a least squares problem on the variables in $\cE$ (in which we consider
$\cE,s$ as fixed):
\begin{equation}
\label{eq:lassols2}
\hbeta^\LARS_\cE(\lambda_k) = \argmin\bigg\{ \|\hbeta_\cE\|_\ltwo : \hbeta_\cE \in 
\argmin_{\beta_\cE \in \R^{|\cE|}} \, 
\|y-(X_\cE\T)^+ \lambda_k s - X_\cE \beta_\cE\|_\ltwo^2 \bigg\}.
\end{equation}
Now we decrease $\lambda$, keeping $\hbeta^\LARS_{-\cE}(\lambda)=0$, and letting 
\begin{equation*}
\hbeta^\LARS_\cE(\lambda)=c-\lambda d,
\end{equation*}
that is, moving in the linear direction suggested by \eqref{eq:lassols}. 
As $\lambda$ decreases, we make two important checks. First, we 
check when (that is, we compute the value of $\lambda$ at which) a variable 
outside the equicorrelation set 
$\cE$ should join $\cE$ because it attains the maximal absolute inner product with 
the residual---we call this the next joining time $\lambda^\mathrm{join}_{k+1}$.
Second, we check when a variable in $\cE$ will have a coefficient path crossing 
through zero---we call this the next crossing time $\lambda^\mathrm{cross}_{k+1}$.

For the first check, for each $i \notin \cE$, we solve the equation 
\begin{equation*}
X_i\T \big(y-X_\cE(c-\lambda d)\big) = \pm \lambda.
\end{equation*} 
A simple calculation shows that the solution is  
\begin{equation}
\label{eq:lassojoin}
t^\mathrm{join}_i = \frac{X_i\T(X_\cE c - y)}
{X_i\T X_\cE d \pm 1} = 
\frac{X_i\T \big(X_\cE(X_\cE)^+ - I\big)y}
{X_i\T (X_\cE\T)^+ s \pm 1},
\end{equation}
called the joining time of the $i$th variable.
(Although the notation is ambiguous, the quantity $t^\mathrm{join}_i$
is uniquely defined, as only one of $+1$ or $-1$ above will yield a value 
in $[0,\lambda_k]$). Hence the next joining time is
\begin{equation}
\label{eq:lassonextj}
\lambda^\mathrm{join}_{k+1} = \max_{i \notin \cE} \, 
t^\mathrm{join}_i,
\end{equation}
and the joining coordinate and its sign are
\begin{equation*}
i^\mathrm{join}_{k+1} = \argmax_i \, t^\mathrm{join}_i 
\;\;\; \text{and} \;\;\;
s^\mathrm{join}_{k+1} = \sign\Big(X_{i^\mathrm{join}_{k+1}}\T 
\big\{y-X\hbeta^\LARS(\lambda^\mathrm{join}_{k+1})\big\}\Big).
\end{equation*}

As for the second check, note that a variable $i \in \cE$ 
will have a zero coefficient when 
$\lambda=c_i/d_i=[(X_\cE)^+y]_i / [(X_\cE\T X_\cE)^+ s]_i$.
Because we are only considering $\lambda \leq \lambda_k$, we
define the crossing time of the $i$th variable as 
\begin{equation}
\label{eq:lassocross}
t^\mathrm{cross}_i = \begin{cases}
\frac{[(X_\cE)^+y]_i}{[(X_\cE\T X_\cE)^+ s]_i}
& \text{if} \;\;
\frac{[(X_\cE)^+y]_i}{[(X_\cE\T X_\cE)^+ s]_i}
\in [0,\lambda_k] \\
0 & \text{otherwise}.
\end{cases}
\end{equation}
The next crossing time is therefore
\begin{equation}\
\label{eq:lassonextc}
\lambda^\mathrm{cross}_{k+1} = \max_{i \in \cE} \, 
t^\mathrm{cross}_i,
\end{equation}
and the crossing coordinate and its sign are
\begin{equation*}
i^\mathrm{cross}_{k+1} = \argmax_i \, t^\mathrm{cross}_i
\;\;\; \text{and} \;\;\;
s^\mathrm{cross}_{k+1} = s_{i^\mathrm{cross}_{k+1}}.
\end{equation*}

Finally, we decrease $\lambda$ until the next joining time or 
crossing time---whichever happens first---by setting $\lambda_{k+1} =
\max\{\lambda^\mathrm{join}_{k+1},\lambda^\mathrm{cross}_{k+1}\}$. 
If $\lambda^{\mathrm{join}}_{k+1} > \lambda^{\mathrm{cross}}_{k+1}$, 
then we add the joining coordinate $i^\mathrm{join}_{k+1}$ to $\cE$ and its 
sign $s^\mathrm{join}_{k+1}$ to $s$. Otherwise, we delete the crossing
coordinate $i^\mathrm{cross}_{k+1}$ from $\cE$ and its sign 
$s^\mathrm{cross}_{k+1}$ from $s$. 

The proof of correctness for this algorithm shows that computed path 
$\hbeta^\LARS(\lambda)$ satisfies the KKT conditions \eqref{eq:lassokkt}
and \eqref{eq:lassosg} at each $\lambda$, and is hence indeed a 
lasso solution path. It also shows that the computed path is continuous at
each knot in the path $\lambda_k$, and hence is globally continuous 
in $\lambda$. The fact that $X_\cE\T X_\cE$ can be singular makes 
the proof somewhat complicated (at least more so than it is for the case
$\rank(X)=p$), and hence we delay its presentation until Appendix 
\ref{app:larsproof}. 

\subsection{Properties of the LARS algorithm and its solutions}
\label{sec:larsprops}

Two basic properties of the LARS lasso path, as mentioned
in the previous section, are piecewise linearity and continuity with respect 
to $\lambda$. The algorithm and the solutions along its computed path possess 
a few other nice properties, most of them discussed in this section, and some
others later in Section \ref{sec:relprops}. We begin with a property of the LARS 
algorithm itself. 

\begin{lemma}
\label{lem:larsbound}
For any $y,X$, the LARS algorithm for the lasso
path performs at most
\begin{equation*}
\sum_{k=0}^p {p \choose k} 2^k = 3^p
\end{equation*} 
iterations before termination. 
\end{lemma}

\begin{proof}
The idea behind the proof is quite simple, and was first noticed by 
\citeasnoun{homotopy1} for their homotopy algorithm: any given pair
of equicorrelation set $\cE$ and sign vector $s$ that appear in one
iteration of the algorithm cannot be revisited in a future iteration, 
due to the linear nature of the solution path.
To elaborate, suppose that $\cE,s$ were the equicorrelation set and 
signs at iteration $k$ and also at iteration $k'$, with $k'>k$. Then this 
would imply that the constraints
\begin{gather}
\label{eq:cnstr1}
\big| X_i\T \big(y-X_\cE \hbeta^\LARS_\cE(\lambda)\big) \big| < \lambda
\;\;\;\text{for each}\;\, i \notin \cE, \\
\label{eq:cnstr2}
s_i \cdot \hbeta^\LARS_\cE(\lambda) > 0 
\;\;\;\text{for each}\;\, i \in \cE, 
\end{gather}
hold at both $\lambda=\lambda_k$ and $\lambda=\lambda_{k'}$. But 
$\hbeta_\cE^\LARS(\lambda)=c-\lambda d$ is a linear function of $\lambda$, 
and this implies that \eqref{eq:cnstr1} and \eqref{eq:cnstr2} also hold at 
every $\lambda \in [\lambda_k',\lambda_k]$, contradicting the fact that 
$k'$ and $k$ are distinct iterations. 
Therefore the total number of iterations performed by the LARS algorithm 
is bounded by the number of distinct pairs of subsets 
$\cE\subseteq \{1,\ldots p\}$ and sign vectors $s \in \{-1,1\}^{|\cE|}$.
\end{proof}

\noindent
{\it Remark.} \citeasnoun{larscomp} showed recently that the upper 
bound for the number of steps taken by the original LARS algorithm, which 
assumes that $\rank(X_\cE)=|\cE|$ throughout the path, can actually be
improved to $(3^p+1)/2$. Their proof is based on the following observation:
if $\cE,s$ are the equicorrelation set and signs at one iteration of 
the algorithm, then $\cE,-s$ cannot appear as the equicorrelation set
and signs in a future iteration. Indeed, this same observation is true for
the extended version of LARS presented here, by essentially the same 
arguments. Hence the upper bound in Lemma \ref{lem:larsbound} can also 
be improved to $(3^p+1)/2$. Interestingly, \citeasnoun{larscomp} further show 
that this upper bound is tight: they construct, for any $p$, a problem instance 
($y$ and $X$) for which the LARS algorithm takes exactly $(3^p+1)/2$ steps.

\smallskip
\smallskip
\smallskip
Next, we show that the end of the LARS lasso solution path 
($\lambda=0$) is itself an interesting least squares solution. 
\begin{lemma}
\label{lem:larslimit}
For any $y,X$, the LARS lasso solution converges
to a minimum $\ell_1$ norm least squares solution as
$\lambda \rightarrow 0^+$, that is,
\begin{equation*}
\lim_{\lambda \rightarrow 0^+} 
\hbeta^{\mathrm{LARS}}(\lambda) = 
\hbeta^{\mathrm{LS},\ell_1},
\end{equation*}
where $\hbeta^{\mathrm{LS},\ell_1} \in 
\argmin_{\beta \in \R^p} \|y-X\beta\|_\ltwo^2$ and 
achieves the minimum $\ell_1$ norm over all such 
solutions. 
\end{lemma}

\begin{proof}
First note that by Lemma \ref{lem:larsbound},
the algorithm always takes a finite number of iterations before
terminating, so the limit here is always attained by the algorithm
(at its last iteration).
Therefore we can write $\hbeta^\LARS(0)=
\lim_{\lambda \rightarrow 0^+} \hbeta^\LARS(\lambda)$.
Now, by construction, the LARS lasso solution satisfies
\begin{equation*}
\big|X_i\T \big(y-X\hbeta^\LARS(\lambda)\big)\big| \leq \lambda 
\;\;\;\text{for each}\;\, i=1,\ldots p,
\end{equation*}
at each $\lambda \in [0,\infty]$. Hence at $\lambda=0$ we have
\begin{equation*}
X_i\T \big(y-X\hbeta^\LARS(0)\big) = 0
\;\;\;\text{for each}\;\, i=1,\ldots p,
\end{equation*}
implying that $\hbeta^\LARS(0)$ is a least squares solution.
Suppose that there exists another least squares solution
$\hbeta^\mathrm{LS}$ with $\|\hbeta^\mathrm{LS}\|_\lone < 
\|\hbeta^\LARS(0)\|_\lone$. Then by continuity of the LARS lasso solution 
path, there exists some $\lambda>0$ such that still $\|\hbeta^\mathrm{LS}\|_\lone 
< \|\hbeta^\LARS(\lambda)\|_\lone$, so that
\begin{equation*}
\half\|y-X\hbeta^\mathrm{LS}\|_\ltwo^2 + 
\lambda\|\hbeta^\mathrm{LS}\|_\lone <
\half\|y-X\hbeta^\LARS(\lambda)\|_\ltwo^2 + 
\lambda\|\hbeta^\LARS(\lambda)\|_\lone.
\end{equation*}
This contradicts the fact that $\hbeta^\LARS(\lambda)$ is a lasso 
solution at $\lambda$, and therefore $\hbeta^\LARS(0)$ achieves the minimum 
$\ell_1$ norm over all least squares solutions.
\end{proof}

We showed in Section \ref{sec:larsalg} that the LARS algorithm constructs 
the lasso solution 
\begin{equation*}
\hbeta^\LARS_{-\cE}(\lambda)=0 \;\;\;\text{and}\;\;\;
\hbeta^\LARS_\cE(\lambda) = (X_\cE)^+\big(y-(X_\cE\T)^+\lambda s\big),
\end{equation*}
by decreasing $\lambda$ from $\infty$, and continually checking whether 
it needs to include or exclude variables from the equicorrelation set $\cE$.
Recall our previous description \eqref{eq:lassosol} of the set of 
lasso solutions at any given $\lambda$. In \eqref{eq:lassosol}, different lasso 
solutions are formed by choosing different vectors $b$ that satisfy the two 
conditions given in \eqref{eq:lassosign}: a null space condition, 
$b \in \nul(X_\cE)$, and a sign condition,
\begin{equation*}
s_i \cdot 
\Big(\big[(X_\cE)^+\big(y-(X_\cE\T)^+\lambda s\big)\big]_i + 
b_i\Big) \geq 0 \;\;\;\text{for}\;\, i \in \cE.
\end{equation*}
We see that the LARS lasso solution corresponds to the choice $b=0$.
When $\rank(X)=|\cE|$, $b=0$ is the only vector in $\nul(X_\cE)$, so it satisfies
the above sign condition by necessity (as we know that a lasso solution must exist 
Lemma \ref{lem:basic}). On the other hand, when $\rank(X)<|\cE|$, it 
is certainly true that $0 \in \nul(X_\cE)$, but it is not at all obvious that 
the sign condition is satisfied by $b=0$. The LARS algorithm establishes this fact by 
constructing an entire lasso solution path with exactly this property ($b=0$) over 
$\lambda \in [0,\infty]$. At the risk of sounding repetitious, we state this result
next in the form of a lemma. 

\begin{lemma}
\label{lem:larssol}
For any $y,X$, and $\lambda > 0$, a lasso solution is given by
\begin{equation}
\label{eq:larssol}
\hbeta^\LARS_{-\cE}=0 \;\;\;\text{and}\;\;\;
\hbeta^\LARS_\cE = (X_\cE)^+\big(y-(X_\cE\T)^+\lambda s\big),
\end{equation}
and this is the solution computed by the
LARS lasso path algorithm. 
\end{lemma}

For one, this lemma is perhaps interesting from a computational point of view:
it says that for any $y,X$, and $\lambda>0$, a lasso solution (indeed, the LARS lasso 
solution) can be computed directly from $\cE$ and $s$, which themselves can be computed
from the unique lasso fit. Further, for any $y,X$, we can start with a lasso solution
at $\lambda>0$ and compute a local solution path using the same LARS steps; see Appendix
\ref{app:loclars} for more details. Aside from computational interests,
the explicit form of a lasso solution given by Lemma \ref{lem:larssol}
may be helpful for the purposes of mathematical analysis; 
for example, this form is used by \citeasnoun{lassodf2} to give a simpler 
proof of the degrees of freedom of the lasso fit, for a general $X$, in terms of the 
equicorrelation set. As another example, it is also used in Section \ref{sec:relprops} 
to prove a necessary condition for the uniqueness of the lasso solution (holding almost 
everywhere in $y$). 

We show in Section \ref{sec:relprops} that, for almost every $y\in\R^n$, 
the LARS lasso solution is supported on all of $\cE$ and hence has the largest support
of any lasso solution (at the same $y,X,\lambda$). 
As lasso solutions all have the same $\ell_1$ norm, by Lemma \ref{lem:basic}, this means 
that the LARS lasso solution spreads out the common $\ell_1$ norm over the largest number 
of coefficients. It may not be surprising, then, that the LARS lasso solution has the 
smallest $\ell_2$ norm among lasso solutions, shown next.

\begin{lemma}
\label{lem:larsminl2}
For any $y,X$, and $\lambda > 0$, the LARS lasso solution $\hbeta^\LARS$ 
has the minimum $\ell_2$ norm over all lasso solutions.
\end{lemma}
\begin{proof}
From \eqref{eq:lassosol}, we can see that any
lasso solution has squared $\ell_2$ norm
\begin{equation*}
\|\hbeta\|_\ltwo^2 = 
\big\|(X_\cE)^+\big(y-(X_\cE\T)^+\lambda s\big)\big\|_\ltwo^2 + \|b\|_\ltwo^2,
\end{equation*}
since $b \in \nul(X_\cE)$. Hence
$\|\hbeta\|_\ltwo^2 \geq \|\hbeta^\LARS\|_\ltwo^2$, 
with equality if and only if $b=0$. 
\end{proof}

Mixing together the $\ell_1$ and $\ell_2$ norms brings 
to mind the elastic net \cite{enet}, which penalizes both 
the $\ell_1$ norm and the squared $\ell_2$ norm of the coefficient vector. The 
elastic net utilizes two tuning parameters $\lambda_1,\lambda_2 \geq 0$ (this 
notation should not to be confused with the knots in the LARS lasso path), and 
solves the criterion\footnote{This is actually what 
\citeasnoun{enet} call the ``naive'' elastic net solution, and the modification
$(1+\lambda_2)\hbeta^\mathrm{EN}$ is what the authors
refer to as the elastic net estimate. But in the limit as $\lambda_2 \rightarrow 0^+$,
these two estimates are equivalent, so our result in Lemma \ref{lem:larsenet} holds for 
this modified estimate as well.}
\begin{equation}
\label{eq:enet}
\hbeta^\mathrm{EN} \in \argmin_{\beta \in \R^p} \, 
\half \|y-X\beta\|_\ltwo^2 + \lambda_1 \|\beta\|_\lone 
+ \frac{\lambda_2}{2} \|\beta\|_\ltwo^2.
\end{equation}
For any $\lambda_2>0$, the elastic net solution 
$\hbeta^{\mathrm{EN}}=\hbeta^{\mathrm{EN}}(\lambda_1,\lambda_2)$
is unique, since the criterion is strictly convex. 

Note that if $\lambda_2=0$, then \eqref{eq:enet} is just the lasso problem. 
On the other hand, if $\lambda_1=0$, then \eqref{eq:enet} reduces 
to ridge regression. It is well-known that the ridge regression 
solution $\hbeta^{\mathrm{ridge}}(\lambda_2)
= \hbeta^\mathrm{EN}(0,\lambda_2)$ converges to the minimum $\ell_2$ norm
least squares solution as $\lambda_2 \rightarrow 0^+$.
Our next result is analogous to
this fact: it says that for any fixed $\lambda_1>0$, the elastic net solution 
converges to the minimum $\ell_2$ norm lasso solution---that is, the LARS 
lasso solution---as $\lambda_2 \rightarrow 0^+$,

\begin{lemma}
\label{lem:larsenet}
Fix any $X$ and $\lambda_1 > 0$. For almost every $y \in \R^n$,
the elastic net solution converges to the LARS lasso solution
as $\lambda_2 \rightarrow 0^+$, that is,
\begin{equation*}
\lim_{\lambda_2 \rightarrow 0^+}
\hbeta^{\mathrm{EN}}(\lambda_1,\lambda_2) = 
\hbeta^\LARS(\lambda_1).
\end{equation*}
\end{lemma}
\begin{proof}
By Lemma \ref{lem:actbig}, we know that for any $y \notin \cN$,
where $\cN \subseteq \R^n$ is a set of measure zero, the LARS lasso at $\lambda_1$
satisfies $\hbeta^\LARS(\lambda_1)_i \not= 0$ for all $i \in \cE$. Hence fix 
$y \notin \cN$. First note that we can rewrite the LARS lasso solution as
\begin{equation*}
\hbeta^\LARS_{-\cE}(\lambda_1)=0 \;\;\;\text{and}\;\;\;
\hbeta^\LARS_\cE(\lambda_1) = (X_\cE\T X_\cE)^+(X_\cE\T y - \lambda_1 s).
\end{equation*}
Define the function 
\begin{align*}
f(\lambda_2) &= (X_\cE\T X_\cE + \lambda_2 I)^{-1}
(X_\cE\T y - \lambda_1 s) \;\;\;\text{for}\;\,\lambda_2>0,\\
f(0) &= (X_\cE\T X_\cE)^+ (X_\cE\T y - \lambda_1 s).
\end{align*}
For fixed $\cE,s$, the function $f$ is continuous on
$[0,\infty)$ (continuity at $0$ can be verified, for example, by looking at 
the singular value decomposition of $(X_\cE\T X_\cE + \lambda_2I)^{-1}$.) 
Hence it suffices to show that for small enough $\lambda_2>0$, the elastic 
net solution at $\lambda_1,\lambda_2$ is given by
\begin{equation*}
\hbeta^{\mathrm{EN}}_{-\cE}(\lambda_1,\lambda_2)=0 \;\;\;\text{and}\;\;\;
\hbeta^{\mathrm{EN}}_\cE(\lambda_1,\lambda_2) = f(\lambda_2).
\end{equation*}

To this end, we show that the above proposed solution satisfies
the KKT conditions for small enough $\lambda_2$. The KKT conditions
for the elastic net problem are
\begin{gather}
\label{eq:enetkkt}
X\T (y-X\hbeta^{\mathrm{EN}}) - \lambda_2 \hbeta^{\mathrm{EN}} 
= \lambda_1 \gamma, \\
\label{eq:enetsg}
\gamma_i \in \begin{cases}
\{\sign(\hbeta^{\mathrm{EN}}_i)\} & \text{if} \;\; 
\hbeta^{\mathrm{EN}}_i \not= 0 \\
[-1,1] & \text{if} \;\; \hbeta^{\mathrm{EN}}_i = 0
\end{cases}, \;\;\;\text{for}\;\, i=1,\ldots p.
\end{gather}
Recall that $f(0)=\hbeta^\LARS_\cE(\lambda_1)$ are the equicorrelation 
coefficients of the LARS lasso solution at $\lambda_1$. As $y \notin \cN$, we have 
$f(0)_i \not= 0$ for each $i \in \cE$, and further, $\sign(f(0)_i)=s_i$ for all 
$i \in \cE$. Therefore the continuity of $f$ implies that for small enough $\lambda_2$,
$f(\lambda_2)_i \not= 0$ and $\sign(f(\lambda_2)_i)=s_i$ for all $i \in \cE$.
Also, we know that $\|X_{-\cE}\T(y-X_\cE f(0))\|_\linf < \lambda_1$ by definition
of the equicorrelation set $\cE$, and again, the continuity of $f$ implies that for
small enough $\lambda_2$, $\|X_{-\cE}\T(y-X_\cE f(\lambda_2))\|_\linf < \lambda_1$.
Finally, direct calculcation shows that
\begin{align*}
X_\cE\T \big(y-X_\cE f(\lambda_2)\big) - \lambda_2 f(\lambda_2) &=
X_\cE\T y - (X_\cE\T X_\cE+\lambda_2 I)(X_\cE\T X_\cE+\lambda_2 I)^{-1} X_\cE\T y \,\,+ \\
&\;\;\;\;\; (X_\cE\T X_\cE+\lambda_2 I)(X_\cE\T X_\cE+\lambda_2 I)^{-1} \lambda_1 s \\
&= \lambda_1 s.
\end{align*}
This verifies the KKT conditions for small enough $\lambda_2$, and completes the 
proof.
\end{proof}

In Section \ref{sec:relprops}, we discuss a few more properties of LARS lasso
solutions, in the context of studying the various support sets of lasso 
solutions. In the next section, we present a simple method for computing lower 
and upper bounds on the coefficients of lasso solutions, useful when the 
solution is not unique.

\section{Lasso coefficient bounds}
\label{sec:coefbounds}

Here we again consider a general predictor matrix $X$ 
(not necessarily having columns in general position), so that the lasso
solution is not necessarily unique. We show that it is possible to compute lower and 
upper bounds on the coefficients of lasso solutions, for any given problem 
instance, using linear programming. We begin by revisiting the KKT conditions.

\subsection{Back to the KKT conditions}
\label{sec:lassoact}

The KKT conditions for the lasso problem were given in \eqref{eq:lassokkt}
and \eqref{eq:lassosg}. Recall that the lasso fit $X\hbeta$ is always unique,
by Lemma \ref{lem:basic}. Note that when $\lambda>0$, we can rewrite
\eqref{eq:lassokkt} as 
\begin{equation*}
\gamma = \frac{1}{\lambda}X\T(y-X\hbeta),
\end{equation*}
implying
that the optimal subgradient $\gamma$ is itself unique. According to its definition 
\eqref{eq:lassosg}, the components of $\gamma$ give the signs of nonzero coefficients of 
any lasso solution, and therefore the uniqueness of $\gamma$ immediately implies the 
following result. 

\begin{lemma}
\label{lem:common}
For any $y,X$, and $\lambda>0$, any two lasso
solutions $\hbeta^{(1)}$ and $\hbeta^{(2)}$ must satisfy
$\hbeta^{(1)}_i \cdot \hbeta^{(2)}_i \geq 0$ for 
$i=1,\ldots p$. In other words, any two lasso solutions
must have the same signs over their common support.
\end{lemma}

In a sense, this result is reassuring---it says that even
when the lasso solution is not necessarily unique, lasso coefficients 
must maintain consistent signs. Note that the same is certainly not true of least
squares solutions (corresponding to $\lambda=0$), which causes problems for 
interpretation, as mentioned in the introduction.
Lemma \ref{lem:common} will be helpful when we derive lasso coefficient bounds 
shortly.

We also saw in the introduction that different
lasso solutions (at the same $y,X,\lambda$) can have 
different supports, or active sets. The previously derived characterization
of lasso solutions, given in \eqref{eq:lassosol} and \eqref{eq:lassosign},
provides an understanding of how this is possible. It helps to rewrite
\eqref{eq:lassosol} and \eqref{eq:lassosign} as
\begin{equation}
\label{eq:lassosol2}
\hbeta_{-\cE} = 0 \;\;\;\text{and}\;\;\;
\hbeta_\cE = \hbeta^\LARS_\cE + b,
\end{equation}
where $b$ is subject to
\begin{equation}
\label{eq:lassosign2}
b \in \nul(X_\cE) \;\;\;\text{and}\;\;\;
s_i \cdot (\hbeta^\LARS_i + b_i) \geq 0,
\;\, i\in \cE,
\end{equation}
and $\hbeta^\LARS$ is the fundamental solution traced by the
LARS algorithm, as given in \eqref{eq:larssol}. Hence for 
for a lasso solution $\hbeta$ to have an active set $\cA=\supp(\hbeta)$, 
we can see that we must have $\cA \subseteq \cE$ and 
$\hbeta_\cE = \hbeta^\LARS_\cE+b$, where $b$ satisfies 
\eqref{eq:lassosign2} and also
\begin{align*}
b_i &= -\hbeta^\LARS_i \;\;\;\text{for}\;\, 
i \notin \cE\setminus\cA,\\
b_i &\not= -\hbeta^\LARS_i \;\;\;\text{for}\;\, 
i \in \cE\setminus\cA.
\end{align*}
As we discussed in the introduction,
the fact that there may be different active sets corresponding to
different lasso solutions (at the same $y,X,\lambda$) is 
perhaps concerning, because different active sets provide different 
``stories'' regarding which predictor variables are important. One 
might ask: given a specific variable of 
interest $i \in \cE$ (recalling that all variables outside of $\cE$ necessarily 
have zero coefficients), is it possible for the $i$th coefficient to be 
nonzero at one lasso solution but zero at another? The answer to this
question depends on the interplay between the constraints in 
\eqref{eq:lassosign2}, and as we show next, it is achieved
by solving a simple linear program.

\subsection{The polytope of solutions and lasso coefficient bounds}
\label{sec:lassopoly}

The key observation here is that the set of lasso solutions 
defined by \eqref{eq:lassosol2} and \eqref{eq:lassosign2} 
forms a convex polytope. Consider writing the set
of lasso solutions as
\begin{equation}
\label{eq:lassopoly} 
\hbeta_{-\cE}=0 \;\;\;\text{and}\;\;\; 
\hbeta_\cE \in 
K = \{x \in \R^{|\cE|} : Px = \hbeta^\LARS_\cE,\, 
Sx \geq 0 \},
\end{equation}
where $P = P_{\mathrm{row}(X_\cE)}$ and $S = \diag(s)$.
That \eqref{eq:lassopoly} is equivalent to
\eqref{eq:lassosol2} and \eqref{eq:lassosign2} follows
from the fact that $\hbeta^\LARS_\cE \in \row(X_\cE)$, 
hence $Px=\hbeta^\LARS_\cE$ if and only if
$x = \hbeta^\LARS_\cE+b$ for some
$b \in \nul(X_\cE)$.

The set $K \subseteq \R^{|\cE|}$ is a polyhedron, since it
is defined by linear equalities and inequalities, and furthermore
it is bounded, as all lasso solutions have the same $\ell_1$ norm
by Lemma \ref{lem:basic}, making it a polytope. The component-wise
extrema of $K$ can be easily computed via linear programming.
In other words, for $i \in \cE$, we can solve the following two linear 
programs:
\begin{align}
\label{eq:lassolo}
\hbeta^\mathrm{lower}_i &=
\min_{x \in \R^{|\cE|}} \, x_i
\;\;\text{subject to}\;\; 
Px = \hbeta^\LARS_\cE, \, Sx \geq 0, \\
\label{eq:lassohi}
\hbeta^\mathrm{upper}_i &=
\max_{x \in \R^{|\cE|}} \, x_i
\;\;\text{subject to}\;\; 
Px = \hbeta^\LARS_\cE, \, Sx \geq 0,
\end{align}
and then we know that the $i$th component of any lasso solution satisfies 
$\hbeta_i \in [\hbeta^\mathrm{lower}_i, \hbeta^\mathrm{upper}_i]$. These bounds are 
tight, in the sense that each is achieved by the $i$th component of some lasso solution
(in fact, this solution is just the minimizer of \eqref{eq:lassolo}, or the maximizer of 
\eqref{eq:lassohi}). By the convexity of $K$, every value between $\hbeta^\mathrm{lower}_i$
and $\hbeta^\mathrm{upper}_i$ is also achieved by the $i$th component of some lasso solution.
Most importantly, the linear programs \eqref{eq:lassolo} and \eqref{eq:lassohi}
can actually be solved in practice. Aside from the obvious dependence on $y,X$, and 
$\lambda$, the relevant quantities $P,S$, and $\hbeta^\LARS_\cE$ only depend on the 
equicorrelation set $\cE$ and signs $s$, which in turn only depend on the unique lasso 
fit. Therefore, one could compute any lasso solution (at $y,X,\lambda$) in order to 
define $\cE,s$, and subsequently $P,S$ and $\hbeta^\LARS_\cE$, all that is
needed in order to solve \eqref{eq:lassolo} and \eqref{eq:lassohi}. We summarize this idea below.

\begin{algorithm}[\textbf{Lasso coefficient bounds}]
\label{alg:coefbounds}
\hfill\par
\smallskip
\smallskip
Given $y,X$, and $\lambda>0$.
\begin{enumerate}
\item Compute any solution $\hbeta$ of the lasso 
problem (at $y,X,\lambda$), to obtain the unique
lasso fit $X\hbeta$.
\item Define the equicorrelation set $\cE$ and
signs $s$, as in \eqref{eq:equiset} and 
\eqref{eq:equisigns}, respectively.
\item Define $P=P_{\mathrm{row}(X_\cE)},\,
S = \diag(s)$, and $\hbeta^\LARS_\cE$
according to \eqref{eq:larssol}.
\item For each $i \in \cE$, compute the coefficient
bounds $\hbeta^\mathrm{lower}_i$ and 
$\hbeta^\mathrm{upper}_i$ by solving the linear
programs \eqref{eq:lassolo} and \eqref{eq:lassohi},
respectively.
\end{enumerate}
\end{algorithm}

Lemma \ref{lem:common} implies a valuable property of the bounding 
interval $[\hbeta^\mathrm{lower}_i,\hbeta^\mathrm{upper}_i]$,
namely, that this interval cannot contain zero in its interior. Otherwise, 
there would be a pair of lasso solutions with opposite signs over the 
$i$th component, contradicting the lemma. Also, we know from Lemma 
\ref{lem:basic} that all lasso solutions have the same $\ell_1$ norm $L$,
and this means that $|\hbeta^\mathrm{lower}_i|,|\hbeta^\mathrm{upper}_i|
\leq L$. Combining these two properties gives the next lemma.

\begin{lemma}
\label{lem:bounds}
Fix any $y,X$, and $\lambda>0$. Let $L$ be
the common $\ell_1$ norm of lasso solutions at $y,X,\lambda$.
Then for any $i \in \cE$, the coefficient bounds 
$\hbeta^\mathrm{lower}_i$ and $\hbeta^\mathrm{upper}_i$ 
defined in \eqref{eq:lassolo} and \eqref{eq:lassohi}
satisfy
\begin{gather*}
[\hbeta^\mathrm{lower}_i,\hbeta^\mathrm{upper}_i]
\subseteq [0,L] \;\;\; \text{if}\;\; s_i > 0, \\
[\hbeta^\mathrm{lower}_i,\hbeta^\mathrm{upper}_i]
\subseteq [-L,0] \;\;\;\text{if}\;\; s_i < 0.
\end{gather*}
\end{lemma}

Using Algorithm \ref{alg:coefbounds}, we can 
identify all variables $i \in \cE$ with one of two categories, based 
on their bounding intervals:
\begin{itemize}
\item[(i)] If $0 \in [\hbeta^\mathrm{lower}_i,\hbeta^\mathrm{upper}_i]$,
then variable $i$ is called {\it dispensable} (to the lasso model at
$y,X,\lambda$), because there is a solution that does not include this
variable in its active set. By Lemma \ref{lem:bounds}, this can only
happen if $\hbeta^\mathrm{lower}_i=0$ or $\hbeta^\mathrm{upper}_i=0$.
\item[(ii)] If $0 \notin [\hbeta^\mathrm{lower}_i,\hbeta^\mathrm{upper}_i]$,
then variable $i$ is called {\it indispensable} (to the lasso model at
$y,X,\lambda$), because every solution includes this variable in its active 
set. By Lemma \ref{lem:bounds}, this can only happen if $\hbeta^\mathrm{lower}_i>0$ 
or $\hbeta^\mathrm{upper}_i<0$.
\end{itemize}

It is helpful to return to the example discussed in the introduction. 
Recall that in this example we took $n=5$ and $p=10$, and for a given $y,X$, and 
$\lambda=1$, we found two lasso solutions: one supported on variables $\{1,2,3,4\}$, 
and another supported on variables $\{1,2,3\}$. In the introduction, we purposely
did not reveal the structure of the predictor matrix $X$;
given what we showed in Section \ref{sec:unique} (that $X$ having columns in 
general position implies a unique lasso solution), it should not be surprising
to find out that here we have $X_4 = (X_2+X_3)/2$. A complete
description of our construction of $X$ and $y$ is as follows:
we first drew the components of the columns $X_1,X_2,X_3$ independently from a standard 
normal distribution, and then defined $X_4 = (X_2+X_3)/2$. We also drew the components 
of $X_5,\ldots X_{10}$ independently from a standard normal distribution, and then 
orthogonalized $X_5,\ldots X_{10}$ with respect to the linear subspace spanned by 
$X_1,\ldots,X_4$. Finally, we defined $y = -X_1+X_2+X_3$. 
The purpose of this construction was to make it easy to detect the relevant variables 
$X_1,\ldots X_4$ for the linear model of $y$ on $X$.

According to the terminology defined above, variable 4 is dispensable to the lasso
model when $\lambda=1$, because it has a nonzero coefficient at one solution but a zero
coefficient at another. This is perhaps not surprising, as $X_2,X_3,X_4$ are linearly 
dependent. How about the other variables? We ran Algorithm \ref{alg:coefbounds} to
answer this question. The results are displayed in Table \ref{tab:smallex}.

\begin{table}[htb]
\begin{center}
\renewcommand*\arraystretch{1.3}
\begin{tabular}{|c|c|c|c|}
\hline
$i$ & $\hbeta^\mathrm{lower}_i$ &
$\hbeta^\LARS_i$ & $\hbeta^\mathrm{upper}_i$ \\
\hline
$1$ & $-0.8928$ & $-0.8928$ & $-0.8928$ \\
\hline
$2$ & $0.2455$ & $0.6201$ & $0.8687$ \\
\hline
$3$ & $0$ & $0.3746$ & $0.6232$ \\
\hline
$4$ & $0$ & $0.4973$ & $1.2465$ \\
\hline
\end{tabular}
\renewcommand*\arraystretch{1}
\caption[tab:smallex]{\small\it The results 
of Algorithm \ref{alg:coefbounds} for the small example
from the introduction, with $n=5$, $p=8$. Shown are the lasso 
coefficient bounds over the equicorrelation set $\cE=\{1,2,3,4\}$.}
\label{tab:smallex}
\end{center}
\end{table}

For the given $y,X$, and $\lambda=1$, the equicorrelation set is $\cE=\{1,2,3,4\}$,
and the sign vector is $s=(-1,1,1,1)\T$ (these are given by running Steps 1 and 2 of 
Algorithm \ref{alg:coefbounds}). Therefore we know that any lasso solution has zero
coefficients for variables $5, \ldots 10$, has a nonpositive first coefficient,
and has nonnegative coefficients for variables $2,3,4$.
The third column of Table \ref{tab:smallex}
shows the LARS lasso solution over the equicorrelation variables. The second and fourth 
columns show the component-wise coefficient bounds $\hbeta^\mathrm{lower}_i$ and 
$\hbeta^\mathrm{upper}_i$, respectively, for $i \in \cE$. We see that variable 3 
is dispensable, because it has a lower bound of zero, meaning that there exists a 
lasso solution that excludes the third variable from its active set (and this solution 
is actually computed by Algorithm \ref{alg:coefbounds}, as it is the minimizer
of the linear program \eqref{eq:lassolo} with $i=3$). 
The same conclusion holds for variable 4. On the other hand, variables 1 and 2 are 
indispensable, because their bounding intervals do not contain zero.

Like variables 3 and 4, variable 2 is linearly dependent on the other variables 
(in the equicorrelation set), but unlike variables 3 and 4, it is indispensable and 
hence assigned a 
nonzero coefficient in every lasso solution. This is the first of a few interesting 
points about dispensability and indispensability, which we discuss below.

\begin{itemize}
\item {\it Linear dependence does not imply dispensability.}
In the example, variable 2 is indispensable, as its 
coefficient has a lower bound of $0.2455>0$,
even though variable 2 is a linear function of variables
3 and 4. Note that in order for the 2nd variable to be
dispensable, we need to be able to use the 
others (variables 1, 3, and 4) to achieve both
the same fit and the same $\ell_1$ norm of the coefficient 
vector. The fact that variable 2 can be written as a linear function
of variables 3 and 4 implies that we can preserve the fit, but not
necessarily the $\ell_1$ norm, with zero weight on variable 2.
Table \ref{tab:smallex} says
that we can make the weight on variable 2 as small 
as $0.2455$ while keeping the fit and the $\ell_1$ norm 
unchanged, but that moving it below $0.2455$ (and maintaining the 
same fit) inflates the $\ell_1$ norm.

\item {\it Linear independence implies indispensability
(almost everywhere).} 
In the next section we show that, given any $X$ and $\lambda$, and
almost every $y\in\R^n$, the quantity $\col(X_\cA)$ is invariant 
over all active sets coming from lasso solutions at $y,X,\lambda$. 
Therefore, almost everywhere in $y$, if variable $i \in \cE$ is linearly 
independent of all $j \notin \cE$ (meaning that $X_i$ cannot be expressed 
as a linear function of $X_j$, $j\notin\cE$), then variable $i$ must be
indispensable---otherwise the span of the active variables 
would be different for different active sets.

\item {\it Individual dispensability does not imply 
pairwise dispensability.} Back to the above example, variables 3 and
4 are both dispensable, but this does not necessarily mean that there 
exists a lasso solution that exludes both 3 and 4 simultaneously from
the active set. Note that the computed solution that achieves a value of zero 
for its 3rd coefficient (the minimizer of \eqref{eq:lassolo} for 
$i=3$) has a nonzero 4th coefficient, and the computed solution 
that achieves zero for its 4th coefficient (the minimizer of 
\eqref{eq:lassolo} for $i=4$) has a nonzero 3rd coefficient. While this suggests
that variables 3 and 4 cannot simultaneously be zero for the current problem, it 
does not serve as definitive proof of such a claim. However, we can check this 
claim by solving \eqref{eq:lassolo}, with $i=4$, subject to the additional
constraint that $x_3=0$. This does in fact yield a positive lower bound, proving
that variables 3 and 4 cannot both be zero at a solution. Furthermore, moving beyond 
pairwise interactions, we can actually enumerate all possible active sets of lasso 
solutions, by recognizing that there is a one-to-one correspondence between active 
sets and faces of the polytope $K$; see Appendix \ref{app:allacts}.
\end{itemize}

Next, we cover some properties of lasso solutions that relate to our work
in this section and in the previous two sections, on uniqueness and non-uniqueness.

\section{Related properties}
\label{sec:relprops}

We present more properties of lasso solutions, relating to issues of
uniqueness and non-uniqueness. The first three sections examine the active sets 
generated by lasso solutions of a given problem instance, when $X$ is a 
general predictor matrix. The results in these three sections are reviewed 
from the literature. In the last section, we give a necessary condition for the 
uniqueness of the lasso solution. 

\subsection{The largest active set}
\label{sec:actbig}

For an arbitrary $X$, recall from Section \ref{sec:coefbounds} that the active set
$\cA$ of any lasso solution is necessarily contained in the equicorrelation set $\cE$. 
We show that the LARS lasso solution has support on all of $\cE$, making it the 
lasso solution with the largest support, for almost every $y \in \R^n$. This result
appeared in \citeasnoun{lassodf2}.

\begin{lemma}
\label{lem:actbig}
Fix any $X$ and $\lambda > 0$. For almost every 
$y \in \R^n$, the LARS lasso solution $\hbeta^\LARS$ has an active set $\cA$ equal 
to the equicorrelation set $\cE$, and therefore achieves the largest active set of
any lasso solution.
\end{lemma}

\begin{proof}
For a matrix $A$, let $A_{[i]}$ denote its $i$th row. Define the set
\begin{equation}
\label{eq:n}
\cN = \bigcup_{\cE,s} \, \bigcup_{i \in \cE} \, \Big\{z \in \R^n :  
\big((X_\cE)^+\big)_{[i]} \big(z - (X_\cE\T)^+ \lambda s\big) = 0\Big\}.
\end{equation}
The first union above is taken over all subsets $\cE \subseteq
\{1,\ldots p\}$ and sign vectors $s \in \{-1,1\}^{|\cE|}$, but implicitly
we exclude sets $\cE$ such that $(X_\cE)^+$ has a row that is entirely
zero. Then $\cN$ has measure zero, because it is a finite union
of affine subspaces of dimension $n-1$.

Now let $y \notin \cN$. We know that no row of $(X_\cE)^+$ can be entirely
zero (otherwise, this means that $X_\cE$ has a zero column, implying
that $\lambda=0$ by definition of the equicorrelation set, contradicting the
assumption in the lemma). Then by construction we have that $\hbeta^\LARS_i 
\not= 0$ for all $i \in \cE$. 
\end{proof}

\noindent
{\it Remark 1.} In the case that the lasso solution is unique, this result 
says that the active set is equal to the equicorrelation set, almost everywhere.

\smallskip
\smallskip
\noindent
{\it Remark 2.} Note that the equicorrelation set $\cE$ (and hence the active set 
of a lasso solution, almost everywhere) can have size
$|\cE|=p$ in the worst case, even when $p>n$. As a trivial example, consider 
the case when $X \in \R^{n\times p}$ has $p$ duplicate columns, 
with $p>n$.

\subsection{The smallest active set}
\label{sec:actsmall}

We have shown that the LARS lasso solution attains the largest possible active 
set, and so a natural question is: what is the smallest possible active set?
The next result is from \citeasnoun{homotopy2} and \citeasnoun{boostpath}.

\begin{lemma}
\label{lem:actsmall}
For any $y,X$, and $\lambda > 0$, there exists
a lasso solution whose set of active variables is
linearly independent. In particular, this means
that there exists a solution whose active set $\cA$
has size $|\cA| \leq \min\{n,p\}$.
\end{lemma}

\begin{proof}
We follow the proof of \citeasnoun{boostpath} closely. Let $\hbeta$ be
a lasso solution, let $\cA=\supp(\hbeta)$ be its active set, and suppose
that $\rank(X_\cA)<|\cA|$. Then by the same arguments as those given 
in Section \ref{sec:unique}, we can write, for some $i\in\cA$,
\begin{equation}
\label{eq:aj}
s_iX_i = \sum_{j \in \cA\setminus\{i\}} a_js_jX_j,
\;\;\;\text{where}\;\;\;
\sum_{j \in \cA\setminus\{i\}} a_j = 1.
\end{equation}
Now define
\begin{equation*}
\theta_i = -s_i \;\;\;\text{and}\;\;\; \theta_j = a_j s_j \;\;\;\text{for}\;\,
j \in \cA\setminus\{i\}.
\end{equation*}
Starting at $\hbeta$, we move in the direction of $\theta$ until a coefficient
hits zero; that is, we define 
\begin{equation*}
\tbeta_{-\cA}=0 \;\;\;\text{and}\;\;\; \tbeta_\cA = \hbeta_\cA + 
\delta \theta,
\end{equation*}
where 
\begin{equation*}
\delta = \min \{ \rho \geq 0 : \hbeta_j + \rho \theta_j = 0 
\;\,\text{for some}\;\, j \in \cA\}.
\end{equation*}
Notice that $\delta$ is guaranteed to be finite, as $\delta \leq |\hbeta_i|$. 
Furthermore, we have $X\tbeta=X\hbeta$ because $\theta\in\nul(X_\cA)$, and also
\begin{align*}
\|\tbeta\|_\lone &= |\tbeta_i| + \sum_{j \in \cA\setminus\{i\}} |\tbeta_j| \\
&= |\hbeta_i| - \delta + \sum_{j \in \cA\setminus\{i\}} (|\hbeta_j|
+ \delta a_j) \\
&= \|\hbeta\|_\lone.
\end{align*}
Hence we have shown that $\tbeta$ achieves the same fit and the same $\ell_1$ 
norm as $\hbeta$, so it is indeed also lasso solution, and it has one fewer nonzero 
coefficient than $\hbeta$. We can now repeat this procedure until we obtain a lasso 
solution whose active set $\cA$ satisfies $\rank(X_\cA)=|\cA|$. 
\end{proof}

\noindent
{\it Remark 1.} This result shows that, for any problem instance, there exists a lasso 
solution supported on $\leq \min\{n,p\}$ variables; some works in the literature have 
misquoted this result by claiming that every lasso solution is supported on 
$\leq \min\{n,p\}$ variables, which is clearly incorrect. When the lasso solution is 
unique, however, Lemma \ref{lem:actsmall} implies that its active set has size 
$\leq \min\{n,p\}$.

\smallskip
\smallskip
\noindent
{\it Remark 2.}
In principle, one could start with any lasso solution, and
follow the proof of Lemma \ref{lem:actsmall} to construct a 
solution whose active set $\cA$ is such that $\rank(X_\cA)=|\cA|$. 
But from a practical perspective, this could be computationally quite difficult,
as computing the constants $a_j$ in \eqref{eq:aj} requires finding a nonzero vector 
in $\nul(X_\cA)$---a nontrivial task that would need to be repeated each time a 
variable is eliminated from the active set. To the best of our 
knowledge, the standard optimization algorithms for the lasso problem (such as
coordinate descent, first-order methods, quadratic programming approaches) 
do not consistently produce lasso solutions with the property that 
$\rank(X_\cA)=|\cA|$ over the active set $\cA$.
This is in contrast to the solution with largest active set, which is computed
by the LARS algorithm.

\smallskip
\smallskip
\noindent
{\it Remark 3.} The proof of Lemma \ref{lem:actsmall} does not actually depend on 
the lasso problem in particular, and the arguments can be extended to cover the general
$\ell_1$ penalized minimization problem \eqref{eq:gen}, with $f$ differentiable and 
strictly convex. (This is in the same spirit as our extension of lasso uniqueness 
results to this general problem in Section \ref{sec:unique}.)
Hence, to put it explicitly, for any differentiable, strictly convex $f$, any $X$, and
$\lambda>0$, there exists a solution of \eqref{eq:gen} whose active set $\cA$ is 
such that $\rank(X_\cA)=|\cA|$. 

\smallskip
\smallskip
\smallskip
The title ``smallest'' active set is justified, because in the next
section we show that the subspace $\col(X_\cA)$ is invariant under all choices of active 
sets $\cA$, for almost every $y\in\R^n$. Therefore, for such $y$, if $\cA$ is an active
set satisfying $\rank(X_\cA)=|\cA|$, then one cannot possibly find a solution whose active 
set has size $<|\cA|$, as this would necessarily change the span of the active variables. 

\subsection{Equivalence of active subspaces}
\label{sec:actequiv}

With the multiplicity of active sets (corresponding to lasso solutions of a given
problem instance), there may be difficulty in identifying and interpreting important
variables, as discussed in the introduction and in Section \ref{sec:coefbounds}. 
Fortunately, it turns out that for almost 
every $y$, the span of the active variables does not depend on the choice of 
lasso solution, as shown in \citeasnoun{lassodf2}. Therefore, even though the linear 
models (given by lasso solutions) may report differences in individual variables, they 
are more or less equivalent in terms of their scope, almost everywhere in $y$.

\begin{lemma}
\label{lem:actequiv}
Fix any $X$ and $\lambda>0$. For almost every $y\in\R^n$, the linear subspace
$\col(X_\cA)$ is exactly the same for any active set $\cA$ coming from a 
lasso solution.
\end{lemma}

Due to the length and technical nature of the proof, we only give a sketch here,
and refer the reader to \citeasnoun{lassodf2} for full details. First, we define
a set $\cN \subseteq \R^n$---somewhat like the set defined in \eqref{eq:n} in 
the proof of Lemma \ref{lem:actbig}---to be a union of affine subspaces of dimension 
$\leq n-1$, and hence $\cN$ has measure zero. Then, for any $y$ except
in this exceptional set $\cN$, we consider any lasso solution at $y$ and examine its
active set $\cA$. Based on the careful construction of $\cN$, we can prove the 
existence of an open set $U$ containing $y$ such
that any $y' \in U$ admits a lasso solution that has an active set $\cA$. In other 
words, this is a result on the local stability of lasso active sets. Next, over $U$, 
the lasso fit can be expressed in terms of the projection map onto $\col(X_\cA)$. The 
uniqueness of the lasso fit finally implies that $\col(X_\cA)$ is the same for any 
choice of active set $\cA$ coming from a lasso solution at $y$. 

\subsection{A necessary condition for uniqueness (almost everywhere)}
\label{sec:uniquenec}

We now give a necessary condition for uniqueness of the lasso solution, that holds 
for almost every $y \in \R^n$ (considering $X$ and $\lambda$ fixed but arbitrary). 
This is in fact the same as the sufficient condition given in Lemma \ref{lem:unique}, 
and hence, for almost every $y$, we have characterized uniqueness 
completely.

\begin{lemma}
\label{lem:unique5}
Fix any $X$ and $\lambda>0$. For almost every $y\in\R^n$, if the lasso solution is
unique, then $\nul(X_\cE)=\{0\}$. 
\end{lemma}

\begin{proof}
Let $\cN$ be as defined in \eqref{eq:n}. Then for $y \notin \cN$, the LARS lasso 
solution $\hbeta^\LARS$ has active set equal to $\cE$. 
If the lasso solution is unique, then it must be the LARS lasso solution. 
Now suppose that $\nul(X_\cE)\not=\{0\}$,
and take any $b \in \nul(X_\cE)$, $b\not=0$. 
As the LARS lasso solution is supported on all
of $\cE$, we know that
\begin{equation*}
s_i \cdot \hbeta^\LARS_i > 0 \;\;\;\text{for all}\;\, i \in \cE.
\end{equation*}
For $\delta>0$, define
\begin{equation*}
\hbeta_{-\cE}=0 \;\;\;\text{and}\;\;\;
\hbeta_\cE = \hbeta^\LARS_\cE + \delta b.
\end{equation*}
Then we know that 
\begin{equation*}
\delta b \in \nul(X_\cE) \;\;\;\text{and}\;\;\;
s_i \cdot (\hbeta^\LARS_i + \delta b_i) > 0, 
\;\, i\in \cE,
\end{equation*}
the above inequality holding for small enough $\delta>0$, by
continuity. Therefore $\hbeta \not= \hbeta^\LARS$ is also a solution, 
contradicting uniqueness, which means that $\nul(X_\cE)=\{0\}$.
\end{proof}

\section{Discussion}
\label{sec:discuss}

We studied the lasso problem, covering conditions for uniqueness,
as well as results aimed at better understanding the behavior of lasso
solutions in the non-unique case. Some of the results presented in this 
paper were already known in the literature, and others were novel. We 
give a summary here. 

Section \ref{sec:unique} showed that any one of the
following three conditions is sufficient for uniqueness of the lasso 
solution: (i) $\nul(X_\cE)=\{0\}$, where $\cE$ is the unique equicorrelation
set; (ii) $X$ has columns in general position; (iii) $X$ has entries drawn 
from a continuous probability
distribution (the implication now being uniqueness with probability one). 
These results can all be found in the literature, in one form or another.
They also apply to a more general $\ell_1$ penalized minimization problem,
provided that the loss function is differentiable and strictly convex when
considered a function of $X\beta$ (this covers, for example, $\ell_1$ 
penalized logistic regression and $\ell_1$ penalized Poisson regression). 
Section \ref{sec:relprops} showed that for the lasso problem, the 
condition $\nul(X_\cE)=\{0\}$ is also necessary for uniqueness of the solution, 
almost everywhere in $y$. To the best of our knowledge, this is a new result.

Sections \ref{sec:lars} and \ref{sec:coefbounds} contained novel work on 
extending the LARS path algorithm to the non-unique case, and on bounding the 
coefficients of lasso solutions in the non-unique case, respectively. 
The newly proposed LARS algorithm works for any predictor matrix $X$, whereas
the original LARS algorithm only works when the lasso solution path is unique. 
Although our extension may superficially appear to be quite minor, its proof of 
correctness is somewhat more involved. In Section \ref{sec:lars} we also 
discussed some interesting properties of LARS lasso solutions in the non-unique 
case.
Section \ref{sec:coefbounds} derived a simple method for computing marginal 
lower and upper bounds for the coefficients of lasso solutions of any given
problem instance. It is also in this section that we showed that no two lasso
solutions can exhibit different signs for a common active variable, implying
that the bounding intervals cannot contain zero in their interiors. 
These intervals allowed
us to categorize each equicorrelation variable as either ``dispensable''---meaning
that some lasso solution excludes this variable from active set, or
``indispensable''---meaning that every lasso solution includes this variable
in its active set. We hope that this represents progress towards interpretation
in the non-unique case. 

Finally, the remainder of Section \ref{sec:relprops} reviewed existing results
from the literature on the active sets of lasso solutions in the non-unique 
case. The first was the fact that the LARS lasso solution is fully supported 
on $\cE$, and hence attains the largest active set, almost everywhere in $y$. 
Next, there always exists a lasso solution whose active set $\cA$ satisfies
$\rank(X_\cA) = |\cA|$, and therefore has size $|\cA| \leq \min\{n,p\}$. 
The last result gave an equivalence between all active sets of lasso
solutions of a given problem instance: for almost every $y$, the subspace 
$\col(X_\cA)$ is the same for any active set $\cA$ of a lasso solution.

\section*{Acknowledgements}

The idea for computing bounds on the coefficients of lasso solutions 
was inspired by a similar idea of Max Jacob Grazier G'Sell, for bounding
the uncertainty in maximum likelihood estimates. We would like to thank
Jacob Bien, Trevor Hastie, and Robert Tibshirani for helpful comments.

\appendix 
\section{Appendix}
\subsection{Proof of correctness of the LARS algorithm}
\label{app:larsproof}
We prove that for a general $X$, the LARS algorithm (Algorithm
\ref{alg:lars}) computes a lasso solution path, by
induction on $k$, the iteration counter. The key result is 
Lemma \ref{lem:larsid}, which shows that the LARS lasso solution
is continuous at each knot $\lambda_k$ in the path, as we change
the equicorrelation set and signs from one iteration to the next.
We delay the presentation and proof of Lemma \ref{lem:larsid} 
until we discuss the proof of correctness, for the sake of clarity.

The base case $k=0$ is straightforward, hence assume that the 
computed path is a solution path 
through iteration $k-1$, that is, for all $\lambda \geq \lambda_k$. 
Consider the $k$th iteration, and let $\cE$ and $s$ denote the current 
equicorrelation set and signs.
First we note that the LARS lasso solution, as defined in terms
of the current $\cE,s$, satisfies the KKT conditions at $\lambda_k$. 
This is implied by Lemma \ref{lem:larsid}, and the fact that the KKT 
conditions were satisfied at $\lambda_k$ with the old equicorrelation 
set and signs. To be more explicit, Lemma \ref{lem:larsid} and the 
inductive hypothesis together imply that
\begin{equation*}
\big\|X_{-\cE}\T \big(y- X\hbeta^\LARS(\lambda_k)\big)\|_\linf < \lambda_k,
\;\;\;
X_\cE\T \big(y- X\hbeta^\LARS(\lambda_k)\big) = \lambda_k s,
\end{equation*}
and $s=\sign(\hbeta^\LARS_\cE(\lambda_k))$, which verifies the KKT conditions
at $\lambda_k$. Now note that for any $\lambda \leq \lambda_k$ 
(recalling the definition of $\hbeta^\LARS(\lambda)$), we have
\begin{align*}
X_\cE\T \big(y- X\hbeta^\LARS(\lambda)\big) &= 
X_\cE\T y - X_\cE\T X_\cE (X_\cE)^+ y + X_\cE\T (X_\cE\T)^+ \lambda s \\
&= X_\cE\T (X_\cE\T)^+ \lambda s \\
&= \lambda s,
\end{align*}
where the last equality holds as $s \in \row(X_\cE)$. Therefore,
as $\lambda$ decreases, only one of the following two conditions can
break: $\|X_{-\cE}\T (y-X\hbeta^\LARS(\lambda)\|_\linf
< \lambda$, or  $s = \sign(\hbeta^\LARS_\cE(\lambda))$. 
The first breaks at the next joining time $\lambda^\mathrm{join}_{k+1}$,
and the second breaks at the next crossing time
$\lambda^\mathrm{cross}_{k+1}$.
Since we only decrease $\lambda$ to
$\lambda_{k+1}=\max\{\lambda^\mathrm{join}_{k+1},
\lambda^\mathrm{cross}_{k+1}\}$, we have hence verified
the KKT conditions for $\lambda \geq \lambda_{k+1}$, completing the
proof.

Now we present Lemma \ref{lem:larsid}, which shows that 
$\hbeta^\LARS(\lambda)$ is continuous 
(considered as a function of $\lambda$) at every knot $\lambda_k$. This 
means that the constructed solution path is also globally continuous, 
as it is simply a linear function between knots. We note that 
\citeasnoun{genlasso-supp} proved a parallel lemma (of the same name) 
for their dual path algorithm for the generalized lasso.

\begin{lemma}[\textbf{The insertion-deletion lemma}]
\label{lem:larsid}
At the $k$th iteration of the LARS algorithm, let $\cE$ and $s$ denote
the equicorrelation set and signs, and let $\cE^*$ and $s^*$ 
denote the same quantities at
the beginning of the next iteration. The two possibilities are:
\begin{enumerate}
\item (Insertion) If a variable joins the equicorrelation set at
  $\lambda_{k+1}$, that is, $\cE^*$ and $s^*$ are formed by 
  adding elements to $\cE$ and $s$, then:
\begin{equation}
\label{eq:ecadd}
\left[\begin{array}{c}
(X_\cE)^+ \big(y - (X_\cE\T)^+ \lambda_{k+1} s\big) 
\smallskip \\
0
\end{array}\right] =
\left[\begin{array}{c}
\Big[(X_{\cE^*})^+ \big(y - 
(X_{\cE^*}\T)^+\lambda_{k+1} s^* \big)
\Big]_{-i^{\mathrm{join}}_{k+1}} 
\smallskip \\
\Big[(X_{\cE^*})^+ \big(y - 
(X_{\cE^*}\T)^+\lambda_{k+1} s^* \big)
\Big]_{i^{\mathrm{join}}_{k+1}} 
\end{array}\right].
\end{equation}
\item (Deletion) If a variable leaves the equicorrelation set at
  $\lambda_{k+1}$, that is,  $\cE^*$ and $s^*$ are formed by
  deleting elements from $\cE$ and $s$, then:
\begin{equation}
\label{eq:ecdel}
\left[\begin{array}{c}
\Big[(X_\cE)^+ \big(y - (X_\cE\T)^+ \lambda_{k+1} s \big)
\Big]_{-i^{\mathrm{cross}}_{k+1}} 
\smallskip \\
\Big[(X_\cE)^+ \big(y - (X_\cE\T)^+ \lambda_{k+1} s \big)
\Big]_{i^{\mathrm{cross}}_{k+1}} 
\end{array}\right] =
\left[\begin{array}{c}
(X_{\cE^*})^+ \big( y - (X_{\cE^*}\T)^+ \lambda_{k+1} s\big)
\smallskip \\ 
0
\end{array}\right].
\end{equation}
\end{enumerate}
\end{lemma}

\begin{proof}
We prove each case separately. The deletion case is actually easier so
we start with this first.

\smallskip
\smallskip
\noindent
{\it Case 2: Deletion.} Let 
\begin{equation*}
\left[\begin{array}{c}
x_1 \\ x_2 
\end{array}\right] =
\left[\begin{array}{c}
\Big[(X_\cE)^+ \big(y - (X_\cE\T)^+ \lambda_{k+1} s \big)
\Big]_{-i^{\mathrm{cross}}_{k+1}} 
\smallskip \\
\Big[(X_\cE)^+ \big(y - (X_\cE\T)^+ \lambda_{k+1} s \big)
\Big]_{i^{\mathrm{cross}}_{k+1}} 
\end{array}\right],
\end{equation*}
the left-hand side of \eqref{eq:ecdel}. By definition, we have $x_2 =
0$ because variable $i^{\mathrm{cross}}_{k+1}$ crosses through zero
at $\lambda_{k+1}$. Now we consider $x_1$. 
Assume without a loss of generality that
$i^{\mathrm{cross}}_{k+1}$ is the last of the equicorrelation
variables, so that we can write
\begin{equation*}
\left[\begin{array}{c}
x_1 \\ x_2 
\end{array}\right] =
(X_\cE)^+ \big(y - (X_\cE\T)^+ \lambda_{k+1} s\big).
\end{equation*}
The point $(x_1,x_2)\T$ is the minimum $\ell_2$ norm solution of the
linear equation:
\begin{equation*}
X_\cE\T X_\cE
\left[\begin{array}{c}
x_1 \\ x_2 
\end{array}\right] =
X_\cE\T y - \lambda_{k+1}s.
\end{equation*}
Decomposing this into blocks, 
\begin{equation*}
\left[\begin{array}{cc}
X_{\cE^*}\T X_{\cE^*} & 
X_{\cE^*}\T X_{i^{\mathrm{cross}}_{k+1}} 
\smallskip \\
X_{i^{\mathrm{cross}}_{k+1}}\T X_{\cE^*} &
X_{i^{\mathrm{cross}}_{k+1}}\T
X_{i^{\mathrm{cross}}_{k+1}}
\end{array}\right]
\left[\begin{array}{c}
x_1 \\ x_2 
\end{array}\right] =
\left[\begin{array}{c}
X_{\cE^*}\T 
\smallskip \\
X_{i^{\mathrm{cross}}_{k+1}}\T
\end{array}\right]
y - \lambda_{k+1}
\left[\begin{array}{c}
s^* \\
s^{\mathrm{cross}}_{k+1}
\end{array}\right].
\end{equation*}
Solving this for $x_1$ gives
\begin{align*}
x_1 &= (X_{\cE^*}\T X_{\cE^*})^+ 
\Big[X_{\cE^*}\T y - \lambda_{k+1} s^* - X_{\cE^*}\T
X_{i^{\mathrm{cross}}_{k+1}} x_2\Big] + b \\
&= (X_{\cE^*})^+ \big(y - (X_{\cE^*}\T)^+ 
\lambda_{k+1} s^*\big) + b,
\end{align*}
where $b \in \nul(X_{\cE^*})$. Recalling that $x_1$ must have minimal 
$\ell_2$ norm, we compute 
\begin{equation*}
\|x_1\|_\ltwo^2 = \Big\|(X_{\cE^*})^+ 
\big( y - (X_{\cE^*}\T)^+ \lambda_{k+1} s^*\big) \Big\|_\ltwo^2
+ \|b\|_\ltwo^2,
\end{equation*}
which is smallest when $b=0$. This completes the proof.

\smallskip
\smallskip
\noindent
{\it Case 1: Insertion.} This proof is similar,
but only a little more complicated. Now we let
\begin{equation*}
\left[\begin{array}{c}
x_1 \\ x_2 
\end{array}\right] =
\left[\begin{array}{c}
\Big[(X_{\cE^*})^+ \big( y - 
(X_{\cE^*}\T)^+\lambda_{k+1} s^*\big)
\Big]_{-i^{\mathrm{join}}_{k+1}} 
\smallskip \\
\Big[(X_{\cE^*})^+ \big( y - 
(X_{\cE^*}\T)^+\lambda_{k+1} s^*\big)
\Big]_{i^{\mathrm{join}}_{k+1}} 
\end{array}\right],
\end{equation*}
the right-hand side of \eqref{eq:ecadd}. Assuming without a loss of
generality that $i^{\mathrm{join}}_{k+1}$ is the largest of the
equicorrelation variables, the point $(x_1,x_2)\T$ is the minimum
$\ell_2$ norm solution to the linear equation:
\begin{equation*}
X_{\cE^*}\T X_{\cE^*}
\left[\begin{array}{c}
x_1 \\ x_2 
\end{array}\right] = 
X_{\cE^*}\T y - \lambda_{k+1} s^*.
\end{equation*}
If we now decompose this into blocks, we get
\begin{equation*}
\left[\begin{array}{cc}
X_\cE\T X_\cE & X_\cE\T X_{i^{\mathrm{join}}_{k+1}} 
\smallskip \\
X_{i^{\mathrm{join}}_{k+1}}\T X_\cE &
X_{i^{\mathrm{join}}_{k+1}}\T
X_{i^{\mathrm{join}}_{k+1}}
\end{array}\right]
\left[\begin{array}{c}
x_1 \\ x_2 
\end{array}\right] = 
\left[\begin{array}{c}
X_\cE\T 
\smallskip \\
X_{i^{\mathrm{join}}_{k+1}}\T
\end{array}\right]
y - \lambda_{k+1}
\left[\begin{array}{c}
s \\
s^{\mathrm{join}}_{k+1}
\end{array}\right].
\end{equation*}
Solving this system for $x_1$ in terms of $x_2$ gives
\begin{align*}
x_1 &= (X_\cE\T X_\cE)^+
\Big[ X_\cE\T y - \lambda_{k+1} s - 
X_\cE\T X_{i^{\mathrm{join}}_{k+1}} x_2 \Big] + b \\
&= (X_\cE)^+ \Big[y - (X_\cE\T)^+ \lambda_{k+1} s - 
X_\cE\T X_{i^{\mathrm{join}}_{k+1}} x_2 \Big] + b,
\end{align*}
where $b \in \nul(X_\cE)$, and as we argued in the deletion case, we
know that $b=0$ in order for $x_1$ to have minimal
$\ell_2$ norm. Therefore we only need to show that $x_2=0$. To do
this, we solve for $x_2$ in the above block system, plug in what
we know about $x_1$, and after a bit of calculation we get 
\begin{equation*}
x_2 = \Big[X_{i^{\mathrm{join}}_{k+1}}\T (I-P) 
X_{i^{\mathrm{join}}_{k+1}}\Big]^{-1}
\Big(X_{i^{\mathrm{join}}_{k+1}}\T 
\Big[(I-P)y + (X_\cE\T)^+ \lambda_{k+1} s  
\Big] - \lambda s^{\mathrm{join}}_{k+1} \Big),
\end{equation*}
where we have abbreviated $P=P_{\col(X_\cE)}$. But the expression
inside the parentheses above is exactly
\begin{equation*}
X_{i^{\mathrm{join}}_{k+1}}\T 
\big(y - X \hbeta^\LARS(\lambda_{k+1})\big) - 
\lambda s^{\mathrm{join}}_{k+1}
= 0,
\end{equation*}
by definition of the joining time. Hence
we conclude that $x_2=0$, as desired, and
this completes the proof. 
\end{proof}

\subsection{Local LARS algorithm for the lasso path}
\label{app:loclars}

We argue that there is nothing special about starting the LARS path algorithm at 
$\lambda=\infty$. Given any solution the lasso problem at $y,X$, and 
$\lambda^*>0$, we can define the unique equicorrelation set $\cE$ 
and signs $s$, as in \eqref{eq:equiset} and \eqref{eq:equisigns}. The LARS lasso 
solution at $\lambda^*$ can then be explicitly constructed as in \eqref{eq:larssol},
and by following the same steps as those outlined in Section \ref{sec:larsalg}, 
we can compute the LARS lasso solution path beginning at $\lambda^*$, for decreasing
values of the tuning parameter; that is, over $\lambda \in [0,\lambda^*]$.

In fact, the LARS lasso path can also be computed in the reverse direction, 
for increasing values of the tuning parameter. Beginning with the LARS lasso
solution at $\lambda^*$, it is not hard to see that in this direction (increasing
$\lambda$) a variable enters the equicorrelation set at the next crossing time---the
minimal crossing time larger than $\lambda^*$, and a variable leaves the equicorrelation 
set at the next joining time---the minimal joining time larger than $\lambda^*$. This is
of course the opposite of the behavior of joining and crossing times in the usual 
direction (decreasing $\lambda$). Hence, in this manner, we can compute the LARS
lasso path over $\lambda \in [\lambda^*,\infty]$.

This could be useful in studying a large lasso problem: if we knew a tuning parameter 
value $\lambda^*$ of interest (even approximate interest), then we could compute a 
lasso solution at $\lambda^*$ using one of the many efficient techniques from convex 
optimization (such as coordinate descent, or accelerated first-order methods), and 
subsequently compute a local solution path around $\lambda^*$ to investigate the behavior
of nearby lasso solutions. This can be achieved by finding the knots to the left and right 
of $\lambda^*$ (performing one LARS iteration in the usual direction and one iteration in 
the reverse direction), and repeating this, until a desired range $\lambda \in 
[\lambda^*-\delta_L,\lambda^*+\delta_R]$ is achieved.

\subsection{Enumerating all active sets of lasso solutions}
\label{app:allacts}

We show that the facial structure of the polytope $K$ in \eqref{eq:lassopoly}
describes the collection of active sets of lasso solutions, almost everywhere
in $y$.

\begin{lemma}
Fix any $X$ and $\lambda>0$. For almost every $y\in\R^n$, 
there is a one-to-one
correspondence between active sets of lasso solutions
and nonempty faces of the polyhedron $K$ defined in 
\eqref{eq:lassopoly}.
\end{lemma}
\begin{proof}
Nonnempty faces of $K$ are sets $F$ of the form 
$F = K \cap H \not= \emptyset$, where $H$ is a supporting hyperplane 
to $K$. If $\cA$ is an active set of a lasso solution, then 
there exists an $x \in K$ such that $x_{\cE\setminus\cA}=0$.
Hence, recalling the sign condition in \eqref{eq:lassosign2}, the 
hyperplane $H_{\cE\setminus\cA} = \{x \in \R^{|\cE|} : u^T x = 0\}$, 
where 
$$u_i = \begin{cases}
s_i & \text{if} \; i \in \cE\setminus\cA \\
0 & \text{if} \; i \in \cA,
\end{cases}$$
supports $K$. Furthermore, we have $F = K \cap H = 
\{x \in K : \sum_{i \in \cE\setminus\cA} s_ix_i = 0\}
= \{x \in K : x_{\cE\setminus\cA} = 0\}$.
Therefore every active set $\cA$ corresponds to a nonempty face $F$ 
of $K$. 

Now we show the converse statement holds, for almost every $y$. 
Well, the facets of $K$ are sets of the form 
$F_i = K \cap \{x \in \R^{|\cE|} : x_i = 0\}$ for some $i \in \cE$.\footnote{This
is slightly abusing the notion of a facet, but the argument here can be made rigorous 
by reparametrizing the coordinates in terms of the affine subspace $\{x \in \R^{|\cE|} : 
Px = \hbeta^\LARS_\cE\}$.} Each nonempty proper
face $F$ can be written as an intersection of facets: $F = \cap_{i \in \mathcal{I}} F_i = 
\{x \in K : x_{\mathcal{I}} = 0\}$, and hence $F$ corresponds to the active set 
$\cA = \cE\setminus\mathcal{I}$. The face $F=K$ corresponds to the equicorrelation set
$\cE$, which itself is an active set for almost every $y \in \R^n$ by Lemma 
\ref{lem:actbig}.
\end{proof}

Note that this means that we can enumerate all possible active sets of lasso solutions,
at a given $y,X,\lambda$, by enumerating the faces of the polytope $K$. This is a 
well-studied problem in computational geometry; see, for example, \citeasnoun{enumpoly}
and the references therein. It is worth mentioning that this could be computationally
intensive, as the number of faces can grow very large, even for a polytope of moderate
dimensions.

\bibliographystyle{agsm}
\bibliography{ryantibs}

\end{document}